\documentclass{amsart}
\pdfoutput=1  %% for ArXiv to use pdflatex

\usepackage{amssymb,mathtools}
\usepackage{graphicx}
\usepackage{esint}
\usepackage{xspace}

\usepackage{tikz}
\usetikzlibrary{math}

\usepackage{macros}
\usepackage{mathrsfs}
\usepackage{hyperref}

\usepackage[all,notheorem]{bi-discrete}

\newcommand{\VQh}{\ensuremath{\mathbb{V}_h^{\mathbb{Q}}}}
\newcommand{\VTh}{\ensuremath{\mathbb{V}_h^{\mathbb{P}}}}
\newcommand{\PiQh}{\ensuremath{\Pi_h^{\mathbb{Q}}}}
\newcommand{\PiTh}{\ensuremath{\Pi_h^{\mathbb{P}}}}

\usepackage[backend=biber,maxbibnames=5,maxalphanames=5,style=alphabetic,bibencoding=utf8,giveninits,doi=true,url=false=true]{biblatex}
\AtBeginBibliography{\small}

\begin{document}

\title[Orthotropic growth]{Numerical approximation of variational problems with orthotropic growth}

\author[A. Kh. Balci]{Anna Kh. Balci}
\address[A. Kh. Balci]{Charles University  in Prague, Department of Mathematical Analysis
Sokolovsk\'a 49/83, 186 75 Praha 8, Czech Republic and University of Bielefeld, Department of Mathematics,  Postfach 10 01 31, 33501 Bielefeld, Germany.}
\email{akhripun@math.uni-bielefeld.de}
\urladdr{http://www.bi-discrete.com}
\thanks{The research of A. Kh. Balci and L. Diening is funded by the Deutsche Forschungsgemeinschaft (DFG, German Research Foundation) --- SFB 1283/2 2021 --- 317210226.}
\thanks{The Research of Anna Kh. Balci was also  supported by Charles University  PRIMUS/24/SCI/020 and Research Centre program No. UNCE/24/SCI/005}

\author[L. Diening]{Lars Diening}
\address[L. Diening]{Department of Mathematics, University of Bielefeld, Postfach 10 01 31, 33501 Bielefeld, Germany.}
\email{lars.diening@math.uni-bielefeld.de}
\urladdr{http://www.bi-discrete.com}

\author[A. J. Salgado]{Abner J. Salgado}
\address{Department of Mathematics, University of Tennessee, Knoxville, TN 37996, USA}
\email{asalgad1@utk.edu}
\urladdr{https://math.utk.edu/people/abner-salgado/}
\thanks{The work of AJS is partially supported by NSF grant DMS-2111228.}



\begin{abstract}
We consider the numerical approximation of variational problems with orthotropic growth, that is those where the integrand depends strongly on the coordinate directions with possibly different growth in each direction.
Under realistic regularity assumptions we derive optimal error estimates. These estimates depend on the existence of an orthotropically stable interpolation operator. Over certain meshes we construct an orthotropically stable interpolant that is also a projection. Numerical experiments illustrate and explore the limits of our theory.
\end{abstract}

\maketitle

%\tableofcontents 

\section{Introduction}
\label{sec:Intro}

The purpose of this work is to study the approximation of variational problems with orthotropic growth. To be precise, let $\Omega \subset \Real^d$, with $d >1$, be a bounded Lipschitz domain. We are interested in the finite element approximation of the solution to the following boundary value problem
\begin{equation}
\label{eq:TheProblem}
%   \begin{dcases}
    - \sum_{i=1}^d \partial_i\left( A_i(\partial_i u) \right) = f, \quad \text{ in } \Omega, \qquad
    u=0, \quad \text{ on } \partial\Omega,
%   \end{dcases}
\end{equation}
where the functions $A_i : \Real \to \Real$ are assumed to satisfy non-standard $\varphi_i$--growth and $\varphi_i$--monotonicity conditions in the sense of \cite{MR2418205}; see Section~\ref{sub:Growth} for further details and examples.
% We comment that solutions of this problem can are minimizers of the variational integral
% \[
%   \calJ(v) = \int_\Omega \left( \sum_{i=1}^d \calA_i(|\partial_i v|) - f v \right)\diff{x}, \qquad A_i(t) = \frac{ \calA_i'(t) }{|t|},
% \]
% over admissible functions that vanish on the boundary.

Apart from the intrinsic interest in the study of problems like \eqref{eq:TheProblem}, they appear as models of materials with anisotropic behavior. For this reason, they have gathered some attention in the literature. References \cite{MR3438591,XiaThesis,MR2483260,MR2978583,MR3048250,MR2813447,MR2503835,MR3097236,MR3119074,MR3169023} have studied questions of existence and uniqueness for \eqref{eq:TheProblem} and related elliptic problems. Parabolic problems with a similar operator have been considered in \cite{MR3073153,MR2861754,MR2305342}. Regularity issues have been discussed in \cite{MR3018035,MR4377996,MR3168616,MR3018035,MR2878481,MR4502898,antonini2023global,MR3798025,MR3749368}. The prototypical example of a problem like \eqref{eq:TheProblem} is the orthotropic $\bfp$--Laplacian, see Example~\ref{ex:OrthoPLapcontd}. References \cite{MR2972430} and \cite{MR2968618,MR3115150} have studied the limiting cases of $p_i \to 1$ and $p_i \to \infty$, respectively.

We immediately comment that, even in the case of the $\bfp$--Laplacian, there is a substantial difference between the problem under consideration and its isotropic counterpart, \ie the classical $p$--Laplacian. To see this we note that the mapping $\bfxi \mapsto |\bfxi |^p$, which defines the $p$--Laplacian, degenerates only at the origin, in the sense that its Hessian is positive definite on all $\Real^d \setminus\{\boldsymbol0\}$. On the other hand, the map $\bfxi \mapsto \|\bfxi \|_{\ell^p}^p$, which generates the $\bfp$--Laplacian, degenerates on an unbounded set. Namely, any hyperplane that contains the origin and is perpendicular to a coordinate axis. The general situation, \ie the mapping that generates \eqref{eq:TheProblem} can only be more degenerate.

In lieu of the interest presented above, it is important to devise approximation methods, say via finite elements, for nonlinear problems with orthotropic structure. In this regard, we recall classical works like \cite{MR0388811,MR0672607} for the $p$--Laplacian, and the improvements and so-called pseudonorm error estimates of \cite{MR1192966}. We mention these as they became the inspiration to consider Orlicz space estimates for this problem in \cite{MR2317830}. Many of the techniques and results in this work can be considered a generalization of \cite{MR2317830} to the orthotropic case.

We organize our presentation as follows. In Section~\ref{sec:Prelim} we establish notation and present the main assumptions over the collections of functions detailing the operator and growth conditions, \ie $\{ A_i \}_{i=1}^d$ and $\{\varphi_i\}_{i=1}^d$, respectively. This is followed by a description of the main properties of the continuous problem that are relevant to our discussion. Section~\ref{sec:FEM} details the finite element scheme and its analysis. We obtain optimal error estimates with realistic regularity assumptions. A crucial step in the derivation of error estimates for the presented method is the construction of an interpolant $\Pi_h$ that is locally and \emph{orthotropically} stable. By this we mean that, for every cell in our mesh $T$ and all $i \in \{1, \ldots, d \}$ we have
\begin{equation}
\label{eq:BNS}
  \dashint_T | \partial_i \Pi_h w | \diff{x} \lesssim \dashint_{\calN_T} |\partial_i w | \diff{x}, \qquad \forall w \in W^{1,1}_0(\Omega),
\end{equation}
where $\calN_T$ is the patch of $T$. Over certain domains and classes of meshes we construct operators that satisfy this property in Section~\ref{sec:interpolant}. Some numerical experiments, aimed to explore the sharpness of our mesh restrictions, are finally discussed in Section~\ref{sec:Numerics}.

\section{Notation and preliminaries}
\label{sec:Prelim}

Throughout the course of our discussion, the relation $A \lesssim B$ will denote $A \leq c B$ for a nonessential constant that may change at each occurrence. $A \gtrsim B$ means that $B \lesssim A$, and $A \eqsim B$ is shorthand for $A \lesssim B \lesssim A$. For $\bfxi \in \Real^d$ its Euclidean norm is $|\bfxi|$.

We shall assume that $\Omega \subset \Real^d$ is a Lipschitz domain. Additional assumptions will be imposed when dealing with discretization. Whenever $U \subset \Real^d$ is open, we denote by $|U|$ its $d$--dimensional Lebesgue measure. For $w \in L^1_{\text{loc}}(\Real^d)$ we define
\[
  \dashint_U w \diff{x} = \frac1{|U|}\int_U w \diff{x}.
\]

\subsection{\Nfuncs{} and shifted \Nfuncs}
\label{sub:Orlicz}

We say that $\Phi : [0,\infty) \to [0,\infty)$ is an \Nfunc{} if it is differentiable and its derivative $\Phi'$ is nondecreasing, right continuous, $\Phi'(t)>0$ for $t>0$, and it satisfies $\Phi'(0) =0$. We observe that every \Nfunc{} is convex. We say that $\Phi$ satisfies the $\Delta_2$ condition ($\Phi \in \Delta_2$) if
\[
  \Phi(2t)\lesssim \Phi(t), \quad \forall t>0,
\]
and set $\Delta_2(\Phi)$ to be the smallest constant for which the inequality above holds. For $\Phi$ an \Nfunc, we set
\[
  (\Phi')^{-1}(t) = \sup\left\{ s \in \Real : \Phi(s) \leq t \right\}
\]
and
\[
  \Phi^*(t) = \int_0^t (\Phi')^{-1}(s) \diff s.
\]
This is again an \Nfunc{} which we call the \emph{complementary} function to $\Phi$. Notice that this pair satisfies the following generalization of Young's inequality
\[
  st \leq \Phi(s) + \Phi^*(t), \quad \forall s,t >0.
\]
Important examples are in our context
\[
  \Phi_p(t) = \frac1p t^p, \quad p>1, \qquad \Phi_{p,\delta}(t) = \frac 1p (\delta^2+t^2)^{\frac p2}, \quad \delta \ge 0.
\]
Notice that $\Phi_p, \Phi_{p,\delta} \in \Delta_2$. We refer the reader to \cite{MR0126722,MR3024912,MR2790542} for a further properties and examples of $N$--functions.

We now turn our attention to so-called \emph{shifted} $N$--functions. Given an \Nfunc{} $\Phi$ with $\Phi,\Phi^* \in \Delta_2$, we define the family $\{\Phi_a\}_{a\geq0}$ of shifted \Nfuncs{} via
\[
  \Phi_a'(t)= \frac{t}{a+t} \Phi'(a+t), \qquad a \geq 0.
\]
It is known that $\sup_{a \geq 0} \Delta_2(\Phi_a) < \infty$. See the Appendices of \cite{MR2418205,MR2317830, DFTW20} for many properties of shifted $N$--functions. We assume that our $N$--functions are uniformly convex in the sense that
\begin{equation}
\label{eq:GrowthNfunc}
  \Phi'(t) \eqsim t \Phi''(t), \quad \forall t>0,
\end{equation}
In this case
\begin{equation}
\label{eq:Shif2ndDeriv}
  \Phi''( |s| + |t|)|s-t| \eqsim \Phi'_{|s|}(|s-t|).
\end{equation}
Finally, we have the following shifted version of Young's inequality: for all $\vare>0$ there is $C_\vare>0$ for which
\begin{equation}
\label{eq:ShiftYoung}
  \Phi_a'(t) s \leq \vare \Phi_a(t) + C_\vare \Phi_a(s),
\end{equation}
for all $t,s,a\geq0$.

\subsection{Growth conditions}
\label{sub:Growth}

For the rest of this discussion we shall assume the following nonstandard growth conditions for the differential operators.

\begin{assume}[growth]
\label{assume:growth}
The family $\{ \varphi_i\}_{i=1}^d \subset \Delta_2 \cap C^2(0,\infty)$, and \eqref{eq:GrowthNfunc} holds for all of them. In addition, the family of functions $\{A_i \}_{i=1}^d \subset C(\Real)$ satisfies $A_i(0) = 0$ and the following monotonicity and growth conditions
\begin{equation}
\label{eq:growthi}
  \begin{aligned}
    \left( A_i( \xi) - A_i(\eta) \right) (\xi - \eta) &\gtrsim \varphi_i''( |\xi| + |\eta| )| \xi - \eta|^2, \\
    | A_i(\xi) - A_i(\eta) | &\lesssim \varphi_i''( |\xi| + |\eta| )|\xi - \eta|,
  \end{aligned}
\end{equation}
for all $\xi, \eta \in \Real$ and $i \in \{1, \ldots, d\}$.
\end{assume}

We comment that, see \cite[Lemma 21]{MR2418205}, for every  $\Phi \in \Delta_2 \cap C^2(0,\infty)$ we can construct $A$ for which \eqref{eq:growthi} holds.

Given $\Phi \in \Delta_2$ we define the \Nfunc{} $\Psi$ via~$\Psi(0)=0$ and 
\begin{equation}
\label{eq:ConstructFlux1}
  \Psi'(t) = \sqrt{t\,\Phi'(t)}.
\end{equation}
Then we have that $\Psi \in\Delta_2$ and $\Psi''(t) \eqsim \sqrt{\Phi''(t)}$ for all $t>0$. It is possible to show then that, if $\{ \varphi_i\}_{i=1}^d$ satisfies Assumption~\ref{assume:growth}, the families $\{\psi_i \}_{i=1}^d$ and $\{V_i\}_{i=1}^d$ defined through \eqref{eq:ConstructFlux1} and
\[
  V_i(\xi ) = \psi_i'( |\xi| )\frac{ \xi }{|\xi|}
\]
will satisfy Assumption~\ref{assume:growth} as well.

\begin{example}[orthotropic $\bfp$--Laplacian]
\label{ex:OrthoPLap}
The main example of functions satisfying Assumption~\ref{assume:growth} is as follows: Let $\{p_i\}_{i=1}^d \subset (1,\infty)$. Define
\[
  A_i(t) = |t|^{p_i-2}t.
\]
The operator in this case, which we call the orthotropic $\bfp$--Laplacian, reads
\[
  -\LAP_{\bfp} w = -\sum_{i=1}^d \partial_i \left( |\partial_i w|^{p_i-2} \partial_i w \right).
\]
In this setting, we have that
\[
  \varphi_i(t) = \frac1{p_i} t^{p_i}, \qquad \psi_i(t) =\frac{2}{p_i+2} t^{\frac {p_i+2}{2}}, \qquad V_i(t) = \abs{t}^{\tfrac{p_i-2}2}t.
\]
\end{example}

Given $\{A_i\}_{i=1}^d$ we define the vector fields $\bfA : \Real^d \to \Real^d$ and $\bfV : \Real^d \to \Real^d$, as having coordinates
\[
  \bfA(\bfxi)_i = A_i(\bfxi_i), \qquad \bfV(\bfxi)_i = V_i(\bfxi_i).
\]
Notice that, with this notation, the equation in \eqref{eq:TheProblem} reads $-\DIV(\bfA(\GRAD u)) = f$.

For families that satisfy Assumption~\ref{assume:growth} the following relations hold.

\begin{prop}[equivalences]
\label{prop:UsefulEquivalence}
Let the families of functions $\{\varphi_i\}_{i=1}^d$ and $\{A_i\}_{i=1}^d$ satisfy Assumption~\ref{assume:growth}. Then, for every $\bfxi, \bfeta \in \Real^d$, we have
\begin{align*}
  \left( \bfA(\bfxi) - \bfA(\bfeta) \right) \cdot \left( \bfxi - \bfeta \right)
    &\eqsim \left| \bfV(\bfxi) - \bfV(\bfeta) \right|^2
    \eqsim \sum_{i=1}^d \varphi_{i,|\xi_i|}(|\bfxi_i - \bfeta_i|) \\
    &\eqsim \sum_{i=1}^d \varphi_i''(|\bfxi_i| + |\bfeta_i| )|\bfxi_i - \bfeta_i |^2
\end{align*}
and
\begin{align*}
  \bigabs{\bfA(\bfxi) - \bfA(\bfeta)} \eqsim \phi_{i,\abs{\xi_i}}'(\abs{\bfxi_i- \bfeta_i}).
\end{align*}
\end{prop}
\begin{proof}
It is sufficient to apply componentwise \cite[Lemma 3]{MR2418205} or \cite[Lemma 41]{DFTW20} and add the results.
\end{proof}

\subsection{The continuous problem}
\label{sub:TheProblem}

Let us collect now some results regarding \eqref{eq:TheProblem} that will be useful for the analysis of our numerical scheme. In what follows, given $\{\varphi_i\}_{i=1}^d$ that satisfies Assumption~\ref{assume:growth}, we set $W^{1,\bvarphi}_0(\Omega)$ to be the closure of $C_0^\infty(\Omega)$ with respect to the norm
\[
  \| w \|_{\bvarphi,\Omega} = \int_\Omega \sum_{i=1}^d \|\partial_i w \|_{\varphi_i,\Omega}, \qquad
  \| w \|_{\varphi_i,\Omega} = \inf\left\{ \lambda > 0: \int_\Omega \varphi_i\left( \frac{|w|}\lambda \right) \diff{x} \leq 1 \right\}.
\]
With this notation, we understand \eqref{eq:TheProblem} in the weak sense: $u \in W^{1,\bvarphi}_0(\Omega)$ is a weak solution to \eqref{eq:TheProblem} if
\[
  \int_\Omega \bfA(\GRAD u) \cdot \GRAD v \diff{x}= \int_\Omega f v \diff{x}, \quad \forall v \in W^{1,\bvarphi}_0(\Omega).
\]

The main properties of problem \eqref{eq:TheProblem} are summarized below.

\begin{thm}[properties of $u$]
\label{thm:PropertiesSoln}
Let $f \in \Linf$, then problem \eqref{eq:TheProblem} has a unique weak solution $u \in W^{1,\bvarphi}_0(\Omega)$.
\end{thm}
\begin{proof}
Existence and uniqueness follows from standard arguments of variational problems.
\end{proof}

In order to carry out an approximation theory for this problem, the following assumption shall be required.

\begin{assume}[regularity]
\label{assume:regularity}
The function $u \in W^{1,\bvarphi}_0(\Omega)$ that solves \eqref{eq:TheProblem} satisfies $\bfV(\GRAD u) \in \bfW^{1,2}(\Omega)$.  
\end{assume}

As mentioned in the Introduction, the operator in question is degenerate on an unbounded set and this makes proving any regularity rather challenging. Nevertheless, we conjecture that this is the correct optimal regularity for this problem. Our reasoning is twofold. First, this is the case for the isotropic problem. Second, for the $\bfp$--Laplacian of Example~\ref{ex:OrthoPLap} this has been shown in \cite{MR3749368,MR3798025}, under some conditions on the exponents $\bfp$ and the right hand side $f$.

\begin{example}[orthotropic $\bfp$--Laplacian (continued)]
\label{ex:OrthoPLapcontd}
Let us specialize these results to the case of Example~\ref{ex:OrthoPLap}. Define the orthotropic Sobolev space
\[
  W^{1,\bfp}_0(\Omega) = \left\{ w \in W_0^{1,1}(\Omega): \partial_i w \in L^{p_i}(\Omega) \right\}, \quad \| w \|_{W^{1,\bfp}_0(\Omega)} = \sum_{i=1}^d \| \partial_i w\|_{L^{p_i}(\Omega)}.
\]
and seek for $u \in W^{1,\bfp}_0(\Omega)$ such that
\[
  \int_\Omega \sum_{i=1}^d |\partial_i u|^{p_i-2} \partial_i u \partial_i v \diff{x} = \int_\Omega f v \diff{x}, \quad \forall v \in W^{1,\bfp}_0(\Omega).
\]
Theorem~\ref{thm:PropertiesSoln} asserts then the existence and uniqueness of solutions.

Notice that even in the case $p_i = p \neq 2$ this really is an orthotropic problem, the operator does not coincide with the classical $p$--Laplacian.
\end{example}

\section{Discretization}
\label{sec:FEM}

Let us now proceed with the approximation of problem \eqref{eq:TheProblem}. As it is clear that the problem is very coordinate dependent we shall, in what follows, assume that $\Omega = (0,1)^d$. We let $\{\Triang \}_{h>0}$ be a quasiuniform family of meshes over $\Omega$. For every $h>0$ we assume that every $T \in \Triang$ (which is  called an element) is closed and isoparametrically equivalent to either the unit simplex or the unit cube. The finite element space is
\[
  \Fespace = \left\{ w \in W^{1,1}_0(\Omega) : w_{|T} \in \calP(T), \ T \in \Triang \right\},
\]
where $\calP(T) = \polP_1$ if the elements are equivalent to a simplex, and $\calP(T) = \polQ_1$ if the equivalence is to a cube. For $T \in \Triang$ its patch is
\[
  \calN_T = \bigcup\left\{ K \in \Triang : T\cap K \neq \emptyset \right\}.
\]
Finally we recall that, as a consequence of quasiuniformity, we have that $|T| \eqsim |\calN_T| \eqsim \diam(T)^d \eqsim h^d$.

To approximate the solution of \eqref{eq:TheProblem} we seek for $u_h \in\Fespace$ such that
\begin{equation}
\label{eq:TheApprox}
    \int_\Omega \bfA( \GRAD u_h ) \cdot \GRAD v_h \diff{x} = \int_\Omega f v_h \diff{x}, \quad \forall v_h \in \Fespace.
\end{equation}

Existence and uniqueness of a solution to \eqref{eq:TheApprox} is as in Theorem~\ref{thm:PropertiesSoln}. The approximation properties of this numerical scheme are as follows.

\begin{prop}[best approximation]
\label{prop:Cea}
Let $u \in W^{1,\bvarphi}_0(\Omega)$ be the weak solution to \eqref{eq:TheProblem} and $u_h \in \Fespace$ solve \eqref{eq:TheApprox}. We have
\begin{equation}
\label{eq:Cea}
  \| \bfV(\GRAD u) - \bfV(\GRAD u_h) \|_\Ldeuxd \lesssim \inf_{v_h \in \Fespace} \| \bfV(\GRAD u) - \bfV(\GRAD v_h) \|_\Ldeuxd.
\end{equation}
\end{prop}
\begin{proof}
The derivation of the best approximation property \eqref{eq:Cea} is a small modification of the arguments in \cite[Lemma 5.2]{MR2317830}. We begin by observing that, in this context, Galerkin orthogonality can be expressed as
\[
  \int_\Omega \left( \bfA(\GRAD u) - \bfA(\GRAD u_h) \right) \cdot \GRAD v_h \diff{x} = 0, \quad \forall v_h \in \Fespace.
\]
We use then Proposition~\ref{prop:UsefulEquivalence}, Galerkin orthogonality, the growth condition \ref{eq:growthi}, and finally \eqref{eq:Shif2ndDeriv} to conclude
\begin{align*}
  \int_\Omega \sum_{i=1}^d \varphi_{i,|\partial_i u|}\left(|\partial_i (u - u_h)|\right) \diff{x}
  & \lesssim \int_\Omega \left( \bfA(\GRAD u) - \bfA(\GRAD u_h) \right) \cdot \GRAD(u -u_h) \diff{x}
  \\
  &= \int_\Omega \left( \bfA(\GRAD u) - \bfA(\GRAD u_h) \right) \cdot \GRAD(u -v_h) \diff{x}
  \\
  % &\lesssim \sum_{i=1}^d \int_\Omega \varphi_{i}''( |\partial_iu|+|\partial_iu_h|)
  %   \left|\partial_i(u-u_h) \right| \left| \partial_i(u - v_h) \right| \diff{x}
  % \\
  &\lesssim \sum_{i=1}^d \int_\Omega \varphi_{i,|\partial_i u|}'\left( |\partial_i(u-u_h)| \right) \left| \partial_i(u - v_h) \right| \diff{x}.
\end{align*}
Next, the shifted version of Young's inequality \eqref{eq:ShiftYoung} shows that
\begin{align*}
  \int_\Omega \varphi_{i,|\partial_i u|}'\left( |\partial_i(u-u_h)| \right) \left| \partial_i(u - v_h) \right| \diff{x} &\leq
  \vare \int_\Omega \varphi_{i,|\partial_i u|}\left( |\partial_i(u-u_h)| \right) \diff{x} \\
  &+ C_\vare\int_\Omega \varphi_{i,|\partial_i u|}\left( |\partial_i(u-v_h)| \right) \diff{x}.
\end{align*}
In conclusion, the previous two estimates and a judicious choice of $\vare$ give us the estimate
\[
  \int_\Omega \sum_{i=1}^d \varphi_{i,|\partial_i u|}\left(|\partial_i (u - u_h)|\right) \diff{x} \lesssim
  \int_\Omega \sum_{i=1}^d \varphi_{i,|\partial_i u|}\left(|\partial_i (u - v_h)|\right) \diff{x}.
\]
Proposition~\ref{prop:UsefulEquivalence} then implies that
\[
  \| \bfV( \GRAD u ) - \bfV(\GRAD u_h) \|_\Ldeuxd^2 \lesssim \| \bfV( \GRAD u ) - \bfV(\GRAD v_h) \|_\Ldeuxd^2,
\]
as claimed.
\end{proof}

To obtain rates of convergence, in addition to regularity, we require the existence of a suitable interpolation operator.

\begin{assume}[interpolant]
\label{assume:interpolant}
There is an interpolation operator $\Pi_h : W^{1,1}_0(\Omega) \to \Fespace$ that is:
\begin{itemize}[leftmargin=0.5cm]
  \item Locally and orthotropically stable, \ie \eqref{eq:BNS} holds.
  
  \item It is locally invariant, in the sense that for every $T \in \Triang$, and all $w \in W^{1,1}_0(\Omega)$ such that $w_{|\calN_T} \in \calP(T)$, then
  \[
    (\Pi_h w)_{|T} = w_{|T}.
  \]
\end{itemize}
\end{assume}

We immediately comment that any projection $\Pi_h : W^{1,1}_0(\Omega) \to \Fespace$ that verifies \eqref{eq:BNS}, satisfies Assumption~\ref{assume:interpolant}. 

We will construct such an interpolant in Section~\ref{sec:interpolant}. A useful property of such interpolant is the following anisotropic Orlicz stability.

\begin{lemma}[orthotropic stability]
\label{lem:OrliczStability}
Assume that $\Fespace$ verifies Assumption~\ref{assume:interpolant}. Then, for every $T \in \Triang$, each $i \in \{ 1, \ldots, d \}$, and all $a\geq 0$, we have
\[
  \dashint_T \varphi_{i,a}(|\partial_i \Pi_h w|) \diff{x} \lesssim \dashint_{\calN_T} \varphi_{i,a}(|\partial_i  w|) \diff{x}, \qquad \forall w \in W^{1,1}_0(\Omega),
\]
where the hidden constants are independent of $a$.
\end{lemma}
\begin{proof}
We begin by observing that, since $\varphi_{i,a} \in \Delta_2$ is increasing, from \eqref{eq:BNS} and Young's inequality we obtain
\[
  \varphi_{i,a}\left( \dashint_T|\partial_i \Pi_h w| \diff{x} \right) \lesssim
  \varphi_{i,a}\left( \dashint_{\calN_T}|\partial_i w| \diff{x} \right) \leq
  \dashint_{\calN_T}\varphi_{i,a}\left(|\partial_i w| \right) \diff{x}.
\]

Next we use that, since $(\Pi_h w)_{|T} \in \calP$, then
\[
  |\partial_i \Pi_h w(x) | \lesssim \dashint_T |\partial_i \Pi_h w | \diff{z}, \quad \forall x \in T.
\]
Therefore,
\[
  \dashint_T \varphi_{i,a}(|\partial_i \Pi_h w|) \diff{x} \lesssim
  \dashint_T \varphi_{i,a}\left(  \dashint_T |\partial_i \Pi_h w | \diff{z} \right) \diff{x}
  = \varphi_{i,a}\left(  \dashint_T |\partial_i \Pi_h w | \diff{z} \right).
\]

Combining the previous estimates, stability follows.
\end{proof}

From the anisotropic Orlicz stability and local invariance of $\Pi_h$ it follows that it has optimal approximation properties.

\begin{prop}[interpolation]
\label{prop:interpolate}
Let $w \in W^{1,\bvarphi}_0(\Omega)$. If $\bfV(\GRAD w) \in \bfW^{1,2}_0(\Omega)$, then we have
\[
  \| \bfV(\GRAD w) - \bfV(\GRAD \Pi_h w) \|_\Ldeuxd \lesssim h \| \GRAD \bfV(\GRAD w) \|_\Ldeuxd,
\]
where the implied constant depends only on the quasiuniformity of $\Triang$.
\end{prop}
\begin{proof}
After localization of the $\bfL^2$--norm, we  realize that we need to consider two cases, depending on whether an element is interior or close to the boundary.

First, if $T \in \Triang$ is an interior element, \ie $\calN_T \cap \partial\Omega = \emptyset$, then we see that it is sufficient to prove that
\begin{equation}
\label{eq:Crucial}
\begin{aligned}
  \dashint_T |\bfV(\GRAD w) - \bfV(\GRAD \Pi_h w) |^2 \diff{x} &\lesssim \inf_{\bfq \in \Real^d} \dashint_{\calN_T} |\bfV(\GRAD w) - \bfV(\bfq) |^2 \diff{x} \\
  &\lesssim h^2 \dashint_{\calN_T}  |\GRAD \bfV(\GRAD w) |^2 \diff{x}.
\end{aligned}
\end{equation}
The second inequality follows upon noticing that $\bfV$ is surjective, so that
\[
  \inf_{\bfq \in \Real^d} \dashint_{\calN_T} |\bfV(\GRAD w) - \bfV(\bfq) |^2 \diff{x} = \inf_{\bfq \in \Real^d} \dashint_{\calN_T} |\bfV(\GRAD w) - \bfq |^2 \diff{x}
\]
and applying the Poincar\'e inequality.

It remains then to prove the first inequality in \eqref{eq:Crucial}. To achieve this, for $p \in \calP(T)$, we estimate
\begin{align*}
  \dashint_T |\bfV(\GRAD w) - \bfV(\GRAD \Pi_h w) |^2 \diff{x} &\lesssim \dashint_T |\bfV(\GRAD w) - \bfV(\GRAD p) |^2 \diff{x} \\
  &+ \dashint_T |\bfV(\GRAD p) - \bfV(\GRAD \Pi_h w) |^2 \diff{x}
  = \mathrm{I} + \mathrm{II}.
\end{align*}
Next, we see that for every $\bfq \in \Real^d$ we can find $p \in \calP(T) \subset \polP_1$ for which $\GRAD p = \bfq$, thus by shape regularity
\[
  \mathrm{I} \lesssim \dashint_{\calN_T} |\bfV(\GRAD u) - \bfV(\bfq) |^2 \diff{x}.
\]
Next, we use Proposition~\ref{prop:UsefulEquivalence}, the invariance property of $\Pi_h$, and Lemma~\ref{lem:OrliczStability} to write
\begin{align*}
  \mathrm{II} &\lesssim \sum_{i=1}^d \dashint_T \varphi_{i,|\bfq_i|}(|\partial_i (\Pi_h u - p) |) \diff{x} =
  \sum_{i=1}^d \dashint_T \varphi_{i,|\bfq_i|}(|\partial_i \Pi_h( u - p) |) \diff{x} \\
  &\lesssim \sum_{i=1}^d \dashint_{\calN_T} \varphi_{i,|\bfq_i|}(|\partial_i ( u - p )|) \diff{x}
  = \sum_{i=1}^d \dashint_{\calN_T} \varphi_{i,|\bfq_i|}(|\partial_i u - \bfq_i|) \diff{x} \\
  &\lesssim \dashint_{\calN_T} |\bfV(\GRAD u) - \bfV(\bfq) |^2 \diff{x},
\end{align*}
where in the last step we applied Proposition~\ref{prop:UsefulEquivalence}.

It remains then to deal with boundary elements, \ie those whose patch touches the boundary. For such elements we exploit the geometry of the domain. We do so by noticing that if $T \in \Triang$ verifies $\Gamma = \calN_T \cap \partial \Omega \neq \emptyset$, then this set must have, for some $\frakj \in \{1, \ldots, d \}$, as normal $\bfe_\frakj$, with $\{\bfe_j\}_{j=1}^d$ being the canonical basis of $\Real^d$. We consider then two cases:
  \begin{itemize}[leftmargin=0.5cm]
    \item ($i\neq \frakj$) Since $w$ may be assumed smooth and $w=0$ on $\Gamma$, we must have $\partial_i w =0$ and $V_i(\partial_i w) = 0$ on $\Gamma$ as well. Thus, the Poincar\'e inequality implies that
    \[
      \int_T |V_i(\partial_i w)|^2 \diff{x} \lesssim h^2 \int_{\calN_T} |\GRAD V_i(\partial_i w)|^2 \diff{x},
    \]
    where we used shape regularity. Next, owing to Lemma~\ref{lem:OrliczStability} and Proposition~\ref{prop:UsefulEquivalence},
    \[
      \int_T | V_i(\partial_i \Pi_h w) |^2 \diff{x} \lesssim \int_T \varphi_i(|\partial_i \Pi_h w|) \diff{x} \lesssim \int_{\calN_T} \varphi_i(|\partial_i w|) \diff{x} \lesssim \int_{\calN_T} | V_i(\partial_i  w) |^2 \diff{x}.
    \]
    Gathering these estimates we see that
    \[
      \int_T \left| V_i(\partial_iw) - V_i(\partial_i\Pi_h w) \right|^2 \diff{x} \lesssim \int_{\calN_T} | V_i(\partial_i w) |^2 \diff{x} \lesssim h^2 \int_{\calN_T} | \GRAD V_i(\partial_i w) |^2 \diff{x}.
    \]

    \item ($i=\frakj$) Let, either, $p = qx_\frakj \in \calP(T)$ or $p = q (1-x_\frakj) \in \calP(T)$, with $q \in \Real$ arbitrary. By symmetry we only consider the first case. We begin by writing
    \begin{align*}
      \int_T \left| V_\frakj(\partial_\frakj w) - V_\frakj (\partial_\frakj  \Pi_h w) \right|^2 \diff{x} &\lesssim \int_T \left| V_\frakj (\partial_\frakj w) - V_\frakj (\partial_\frakj  p) \right|^2 \diff{x} \\
      &+ \int_T \left| V_\frakj (\partial_\frakj p) - V_\frakj (\partial_\frakj  \Pi_h w) \right|^2 \diff{x}
      = \frakA + \frakB,
    \end{align*}
    and estimate each term separately.
    First, since $\partial_\frakj p = q$ is arbitrary, it can be chosen to obtain
    \begin{align*}
      \frakA &= \int_T \left| V_\frakj (\partial_\frakj w) - V_\frakj (\partial_\frakj p) \right|^2 \diff{x} = \int_T \left| V_\frakj (\partial_\frakj w) - V_\frakj (q) \right|^2 \diff{x} \\
      &= \int_T \left| V_\frakj (\partial_\frakj w) - Q \right|^2 \diff{x}
      \lesssim h^2 \int_T \left| \GRAD V_\frakj (\partial_\frakj w) \right|^2 \diff{x},
    \end{align*}
    where we also used that $V_\frakj $ is surjective. With such a chosen $q$ we continue and treat the term $\frakB$.

    Since, by construction, $p_{|\Gamma}= 0$ we have that $(\Pi_h p)_{|T}=p$. We can now use Proposition~\ref{prop:UsefulEquivalence} to get
    \begin{align*}
      \frakB &\lesssim \int_T \varphi_{\frakj ,|q|}(|\partial_\frakj  \Pi_h w - q|) \diff{x} =
      \int_T \varphi_{\frakj ,|q|}(|\partial_\frakj \Pi_h w - \partial_\frakj  p|) \diff{x} \\
      &= \int_T \varphi_{\frakj ,|q|}(|\partial_\frakj \Pi_h (w - p)|) \diff{x}
      \lesssim \int_T \varphi_{\frakj ,|q|}(|\partial_\frakj  (w - p)|) \diff{x} \\
      &\lesssim \int_T |V_\frakj (\partial_\frakj  w) - V_\frakj (\partial_\frakj  p) |^2 \diff{x} = \frakA.
    \end{align*}
  \end{itemize}

Adding over $T \in \Triang$, and using shape regularity, the claimed interpolation estimate follows.
\end{proof}

With the aid of these interpolation estimates, let us now obtain error estimates.

\begin{thm}[rate of convergence]
\label{thm:Rates}
Let $u \in W^{1,\bvarphi}_0(\Omega)$ be the weak solution to \eqref{eq:TheProblem} and $u_h \in \Fespace$ solve \eqref{eq:TheApprox}. If Assumption~\ref{assume:regularity} holds and $\Fespace$ verifies Assumption~\ref{assume:interpolant}, then we have the optimal error estimate
\begin{equation}
\label{eq:Rate}
  \| \bfV(\GRAD u) - \bfV(\GRAD u_h) \|_\Ldeuxd \lesssim h \| \GRAD \bfV(\GRAD u) \|_\Ldeuxd.
\end{equation}
\end{thm}
\begin{proof}
By assumption $\GRAD \bfV(\GRAD u) \in \Ldeuxd$. Therefore, the interpolation estimates of Proposition~\ref{prop:interpolate} hold. We merely then need to combine this with Proposition~\ref{prop:Cea}.
\end{proof}

\section{Suitable interpolation operators}
\label{sec:interpolant}

The purpose of this section is to explore the possibility of constructing interpolants that verify Assumption~\ref{assume:interpolant}. As stated before, problem \eqref{eq:TheProblem} itself is strongly dependent on the chosen coordinate system. Thus, we shall continue to assume that $\Omega = (0,1)^d$.

In order to describe the types of meshes that we allow we introduce some notation. We let $N \in \polN$ and set the mesh size to be $h = \tfrac1N$. Define $\bfN = [0,N]^d \cap \polN_0^d$ and $\bfN^0 = \bfN \cap \Omega$. The nodes of the mesh are
\[
  \vertex(\bfk) = h\bfk, \quad \bfk \in \bfN.
\]
The interior nodes are those indexed by $\bfN^0$. The two types of meshes that will suit our purposes, and the corresponding finite element spaces, are given next.

In the first case, we let $\Triang = \calQ_h = \{ Q_\bfk \}$ where, for $\bfk \in \bfN$ such that $\bfk + \bfe_i \in \bfN$ for all $i \in \{1, \ldots, d\}$, the elements are cubes $Q_\bfk$:
\[
  Q_{\bfk} = \left\{ x \in \bar\Omega : h \bfk\cdot \bfe_i \leq x_i \leq h(\bfk\cdot\bfe_i+1), \ \bfk \in \bfN, \ i \in \{ 0,\ldots, d\} \right\}.
\]
The ensuing finite element space is denoted by $\Fespace^\polQ$. For $\bfk \in \bfN^0$, by $\Lambda_{h,\bfk}^\polQ \in \polQ_1$ we denote the Lagrange basis corresponding to the interior node $\vertex(\bfk)$. Given the special structure of the mesh, if $\Lambda_{h,\bfk}^\polQ$ is the Lagrange basis function associated to node $\vertex(\bfk)$, then we have
\[
  \Lambda_{h,\bfk}^\polQ(x) = \prod_{i=1}^d \lambda_{k_i}(x_i) ,
\]
where $\lambda_{k_m}$ is the one dimensional, piecewise linear, Lagrange basis function associated to the node $k_mh$ over a mesh of size $h$.

The second case corresponds to a simplicial mesh. We set $\Triang = \calT_h = \{ S \}$, where each $S \in \calT_h$ is a simplex. These simplices are obtained by taking every $Q_\bfk \in \calQ_h$ and subdividing it by a Kuhn partition of $d!$ Kuhn simplices. The orientation of the diagonal does not matter, so in two dimensions we can have all diagonals parallel or a -cross pattern. In this case, the finite element space is denoted by $\Fespace^\polP$ and the Lagrange basis functions are $\{ \Lambda_{h,\bfk}^\polP \}_{\bfk \in \bfN^0} \subset \VTh$.

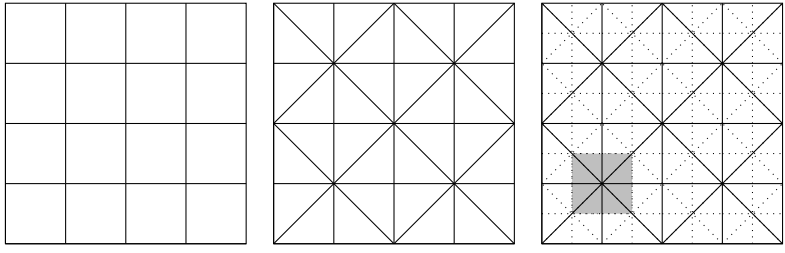
\begin{figure}
  \centering
  
  \begin{tikzpicture}[scale=0.8]
    \draw (0,0) rectangle (4,4);

    \foreach \i in {0,1,2,3}{
      \foreach \j in {0,1,2,3}{
        \begin{scope}[shift={(\i,\j)}]
          \draw (0,0) -- (0,1);
          \draw (0,0) -- (1,0);
        \end{scope}
      }
    }
  \end{tikzpicture}%
  \quad
  \begin{tikzpicture}[scale=0.8]
    \draw (0,0) rectangle (4,4);

    \tikzmath{int \s1; int \s2; \s1=1; \s2=1;}
    \foreach \i in {0,1,2,3}{
      \foreach \j in {0,1,2,3}{
        \begin{scope}[shift={(\i,\j)}]
          \tikzmath{\ijsum = \i+\j;}
          \draw (0,0) -- (0,1);
          \draw (0,0) -- (1,0);
%          \draw (0,0) -- (1,1);
          \draw (0,{Mod(\i+\j,2)}) -- (1, {Mod(\i+\j+1,2});
        \end{scope}
      }
    }

  \end{tikzpicture}%
  \quad
  \begin{tikzpicture}[scale=0.8]
    \draw (0,0) rectangle (4,4);
    
    \draw[dotted, stroke = gray!50, fill = gray!50, very thin] (0.5,0.5) rectangle (1.5, 1.5);
    
    \foreach \i in {0,1,2,3}{
      \foreach \j in {0,1,2,3}{
        \begin{scope}[shift={(\i,\j)}]
          \tikzmath{\ijsum = \i+\j;}
          \draw (0,0) -- (0,1);
          \draw (0,0) -- (1,0);
          \draw (0,{Mod(\i+\j,2)}) -- (1, {Mod(\i+\j+1,2});
        \end{scope}
      }
    }
    \foreach \i in {0,1,2,3}{
      \foreach \j in {0,1,2,3}{
        \begin{scope}[shift={(\i,\j)}]
          \draw[dotted] (0,1/2) -- (1,1/2);
          \draw[dotted] (1/2,0) -- (1/2,1);
          \draw[dotted] (1,0) -- (0,1);
          \draw[dotted] (0,0) -- (1,1);
        \end{scope}
      }
    }
  \end{tikzpicture}

  \caption{The meshes $\mathcal{Q}_h$, $\mathcal{T}_h$, and $\mathcal{T}_{h/2}$; and a sample cell $\sigma_\bfk$.}
  \label{fig:meshes}
\end{figure}

The types of admissible meshes are depicted in Figure~\ref{fig:meshes}.

\subsection{A positivity preserving interpolant}
\label{sub:BNS}
Here we, essentially, show that the interpolant constructed in \cite{MR3378839} is orthotropically stable, \ie it satisfies \eqref{eq:BNS}. We begin by recalling its construction. Given $w \in L^1(\Omega)$ we define
\[
  W_{\bfk} = \dashint_{Q^\circ} w( x + \vertex(\bfk) ) \diff{x}, \qquad Q^\circ = \left( -\frac{h}2, \frac{h}2 \right)^d, \qquad \bfk\in \bfN^0.
\]

\subsubsection{Quadrilateral meshes}
The interpolant in this case is
\begin{equation}
\label{eq:QuadInterpolant}
  \Pi_h^{\polQ,+} w(x) = \sum_{\bfk \in \bfN^0} W_{\bfk} \Lambda_{h,\bfk}^\polQ(x).
\end{equation}
Notice that, by construction, the interpolant preserves homogeneous boundary conditions. Additional properties are given in the following result.

\begin{prop}[interpolant]
\label{prop:QuadStable}
The interpolant defined in \eqref{eq:QuadInterpolant} is positivity preserving and stable in every $L^p(\Omega)$, $p \in [1,\infty]$. In addition, it satisfies Assumption~\ref{assume:interpolant} with an implicit constant of one, \ie for every $w \in W^{1,1}_0(\Omega)$ and all $Q_\bfk \in \calQ_h$, we have
\[
  \dashint_{Q_\bfk} | \partial_i \Pi_h^{\polQ,+} w | \diff{x} \leq \dashint_{\calN_{Q_\bfk}} |\partial_i w | \diff{x}.
\]
\end{prop}
\begin{proof}
We begin by noticing that the positivity preserving property, and the stability for $p \in \{1,\infty\}$ immediately follow from the fact that we are taking local averages and $0 \leq \Lambda_{h,\bfk}^\polQ \leq 1$. The stability for $p \in (1,\infty)$ follows by function space interpolation. We refer the reader to \cite[Lemma 3.5]{MR3378839} for details.

It remains to prove the orthotropic stability. For simplicity we present the proof in the case $d=2$. The arguments can be readily extended to the general case. We prove each item in Assumption~\ref{assume:interpolant} separately.
\begin{enumerate}[label={\arabic{*}.},leftmargin=0.5cm]
  \item We make minor modifications to the proof of \cite[Theorem 3.7]{MR3378839}. Fix $i=1$, as the other coordinate is treated similarly. Let $\bfk \in \bfN$, $Q_{\bfk} \in \calQ_h$, and $z \in Q_{\bfk}$. Notice that, owing to the Cartesian structure of the mesh, we have that $\bfk = (k,l)$ with $k,l \in \{0, \ldots, N\}$, and
  \[
    \partial_1 \Lambda_{h,(k,l)}^\polQ(z) = \lambda_k'(z_1) \lambda_l(z_2) = \pm \frac1h \lambda_l(z_2).
  \]
  In addition, since $\lambda_l(z_2) + \lambda_{l+1}(z_2) =1$,
  \[
    \partial_1\Pi_h^{\polQ,+} w(z) = \frac{ W_{(k+1,l)}-W_{(k,l)} }h \lambda_l(z_2) + \frac{ W_{(k+1,l+1)}-W_{(k,l+1)} }h (1-\lambda_l(z_2)).
  \]
  The definition of $W_{(k,l)}$ gives
  \begin{align*}
    W_{(k+1,l)}-W_{(k,l)} &= \frac1{h^2} \int_{Q^\circ} \left( w(x + \vertex(k+1,l) ) - w(x + \vertex(k,l) ) \right) \diff{x} \\ &= \frac1h \int_{Q^\circ} \int_0^1 \partial_1w( s\vertex(k+1,l) + (1-s)\vertex(k,l) + x ) \diff{s} \diff{x}.
  \end{align*}
  Finally, observe that
  \[
    \dashint_{Q_\bfk} \lambda_l(z_2) \diff{z} = \dashint_{Q_\bfk} (1-\lambda_l(z_2)) \diff{z} = \frac12.
  \]

  With these preparations we estimate
  \begin{align*}
   \dashint_{Q_\bfk} |\partial_1 \Pi_h^{\polQ,+} w|\diff{x} &\leq \frac12 \dashint_{Q^\circ} \int_0^1 \left|\partial_1w( s\vertex(k+1,l) + (1-s)\vertex(k,l) + x )\right| \diff{s} \diff{x} \\
   &+ \frac12 \dashint_{Q^\circ} \int_0^1 \left|\partial_1w( s\vertex(k+1,l+1) + (1-s)\vertex(k,l+1) + x )\right| \diff{s} \diff{x}.
  \end{align*}
  Consider now the first term,
  \begin{multline*}
    \dashint_{Q^\circ} \int_0^1 \left|\partial_1w( s\vertex(k+1,l) + (1-s)\vertex(k,l) + x )\right| \diff{s} \diff{x} \leq \\
    \frac1{h^2} \int_0^1 \int_{\left(k-\tfrac12\right)h}^{\left(k+\tfrac32\right)h} \int_{\left(l-\frac12\right)h}^{\left(l+\tfrac32\right)h} \left|\partial_1 w(x) \right| \diff{x_2} \diff{x_1} \diff{s}
    \leq \dashint_{\calN_{Q_\bfk}} \left|\partial_1 w(x) \right| \diff{x},
  \end{multline*}
  from which the result follows.

  \item We need now to prove the local invariance. This is dealt with in three separate cases.
  \begin{itemize}[leftmargin=0cm]
    \item Let $Q_{\bfk} \in \calQ_h$ be interior, \ie $\calN_{Q_\bfk} \cap \partial\Omega = \emptyset$. In this case, we observe that, owing to the symmetry in our constructions, the interior invariance property holds trivially.
    
    \item Next, if $Q_\bfk \cap \partial\Omega = \emptyset$ but $\calN_{Q_\bfk} \cap \partial\Omega \neq \emptyset$ then we see that the sets where the local averages are computed, which define $W_\bfk$, are well separated from the boundary. Consequently, if $w_{|\calN_{Q_\bfk}} \in \polQ_1$ we have $(\Pi_h^{\polQ,+} w)_{|Q_\bfk} = w_{|Q_\bfk}$.
    
    \item The remaining case, \ie $Q_\bfk \cap \partial\Omega \neq \emptyset$, was already partially dealt with in Proposition~\ref{prop:interpolate}. To see this, in the setting of Proposition~\ref{prop:interpolate}, set $d=2$ and $\frakj = 1$. This means that
    \[
      \Gamma = \calN_T \cap \partial\Omega = \{0\} \times \left[ \vertex(0,l-1) \cdot \bfe_2,  \vertex(0,l+2) \cdot \bfe_2 \right] = h \left( \{0\} \times [l-1,l+2]\right).
    \]
    Consider now the polynomial $p = q x_1 \in \polQ_1$, which vanishes on $\Gamma$. Now, if the intersection of $Q_\bfk$ with the boundary is not empty, in our established notation we must have $Q_\bfk = Q_{(0,l)}$, and then
    \begin{align*}
      (\Pi_h^{\polQ,+} p)_{|Q_{(0,l)}} &= P_{(1,l)} \Lambda_{h,(1,l)}^\polQ + P_{(1,l+1)} \Lambda_{h,(1,l+1)}^\polQ \\
        &= \lambda_1(x_1) \left[ P_{(1,l)} \lambda_2(x_2) + P_{(1,l+1)} (1-\lambda_2(x_2) )\right].
    \end{align*}
    Use now that
    \[
      \lambda_1(x_1) = \frac{x_1}h, \quad \lambda_2(x_2) = \frac{x_2 -hl}h, \quad P_{(1,l)} = P_{(1,l+1)} = qh
    \]
    to conclude that $(\Pi_h^{\polQ,+} p)_{|Q_{(0,l)}}=p$. Notice that we used that the patches used to compute $P_{(1,l)}$ and $P_{(1,l+1)}$ do not touch the boundary.

  \end{itemize}
\end{enumerate}
All properties have been shown and thus the constructed interpolant satisfies Assumption~\ref{assume:interpolant}.
\end{proof}

\subsubsection{Simplicial meshes}
\label{subsub:Triangles}

In this case the nodal values are defined as before, \ie the interpolant is defined as
\[
  \Pi_h^{\polP,+} w(x) = \sum_{\bfk \in \bfN^0} W_{\bfk} \Lambda_{h,\bfk}^\polP(x).
\]

\begin{prop}[stability]
The operator $\Pi_h^{\polP,+}$ is positivity preserving and stable in $L^p(\Omega)$ for $p \in [1,\infty]$. In addition, it satisfies Assumption~\ref{assume:interpolant} where \eqref{eq:BNS} is satisfied with constant one.
\end{prop}
\begin{proof}
The proof repeats the proof of Proposition~\ref{prop:QuadStable}. To see this, it suffices to observe that every $S \in \calT_h$ has $d$ edges that are parallel to the coordinate axes. This allows to construct the needed differences. Similar calculations are carried out in the course of the proof of Proposition~\ref{prop:OrthoStabPKuhnSimplices} below.
\end{proof}

\subsection{A projection}
\label{sub:suitable-projection}

The interpolants of Section~\ref{sub:BNS} are not projections. Here, we construct interpolants that are projections.

As an auxiliary tool we shall also need $\mathcal{T}_{h/2}$. This is the mesh that is obtained if every Kuhn simplex $S \in \mathcal{T}_h$ is fully refined, \ie we perform $d$ consecutive refinements. In the triangulation $\calT_{h/2}$, each $\vertex(\bfk)$ with $\bfk \in \bfN^0$ is surrounded by $2^d$ Kuhn cubes that consist each of $d!$ Kuhn simplices. All diagonals of these adjacent Kuhn cubes contain the point $\vertex(\bfk)$. For each $\bfk \in \bfN^0$ let $\sigma_\bfk$ denote the union of all simplices $S \in \mathcal{T}_{h/2}$ that contain the point $\vertex(\bfk)$. By construction of $\mathcal{T}_{h/2}$ each $\sigma_\bfk$ contains $2^d$ Kuhn cubes, each consisting of $d!$ Kuhn simplices. Note that the patches $\sigma_\bfk$ are translates of each other. This is important for the orthotropic stability. Refer, again, to Figure~\ref{fig:meshes} for a depiction of this construction.

% \begin{figure}[ht!]
%   \centering
%   
%   \begin{tikzpicture}[scale=0.8]
%     \draw (0,0) rectangle (4,4);
% 
%     \foreach \i in {0,1,2,3}{
%       \foreach \j in {0,1,2,3}{
%         \begin{scope}[shift={(\i,\j)}]
%           \draw (0,0) -- (0,1);
%           \draw (0,0) -- (1,0);
%         \end{scope}
%       }
%     }
%   \end{tikzpicture}%
%   \quad
%   \begin{tikzpicture}[scale=0.8]
%     \draw (0,0) rectangle (4,4);
%     
%     \tikzmath{int \s1; int \s2; \s1=1; \s2=1;}
%     \foreach \i in {0,1,2,3}{
%       \foreach \j in {0,1,2,3}{
%         \begin{scope}[shift={(\i,\j)}]
% \tikzmath{\ijsum = \i+\j;}
%           \draw (0,0) -- (0,1);
%           \draw (0,0) -- (1,0);
%           \draw (0,0) -- (1,1);
% %          \draw (0,{Mod(\i+\j,2)}) -- (1, {Mod(\i+\j+1,2});
%         \end{scope}
%       }
%     }
% 
%   \end{tikzpicture}%
%   \quad
%   \begin{tikzpicture}[scale=0.8]
%     \draw (0,0) rectangle (4,4);
%     
%     \foreach \i in {0,1,2,3}{
%       \foreach \j in {0,1,2,3}{
%         \begin{scope}[shift={(\i,\j)}]
%           \draw (0,0) -- (0,1);
%           \draw (0,0) -- (1,0);
%           \draw (0,0) -- (1,1);
%         \end{scope}
%       }
%     }
%     \foreach \i in {0,1,2,3}{
%       \foreach \j in {0,1,2,3}{
%         \begin{scope}[shift={(\i,\j)}]
%           \draw[dotted] (0,1/2) -- (1,1/2);
%           \draw[dotted] (1/2,0) -- (1/2,1);
%           \draw[dotted] (1,0) -- (0,1);
%         \end{scope}
%       }
%     }
%   \end{tikzpicture}
% 
%   \caption{The meshes $\mathcal{Q}_h$, $\mathcal{T}_h$ and $\mathcal{T}_{h/2}$}.
%   \label{fig:meshes}
% \end{figure}

Let $\calP_2(\mathcal{T}_{h/2})$ denote the set of continuous, piecewise quadratic functions on $\mathcal{T}_{h/2}$. Observe that $\Fespace^\polQ \subset \mathcal{P}_2(\mathcal{T}_{h/2})$ and $\Fespace^\polP \subset \mathcal{P}_2(\mathcal{T}_{h/2})$. 
% We index the nodes of $\calT_{h/2}$ via $\bfM \subset \polN_0^d$, and the interior nodes via $\bfM^0$.

Fix $\bfj \in \bfN^0$. Let $\calP_2(\sigma_\bfj)$ denote the restriction of $\calP_2(\calT_{h/2})$ onto $\sigma_\bfj$. Notice that $\dim(\calP_2(\sigma_\bfj)) = 5^d$. We let $\{\rho_{\bfj,k}\}_{k=1}^{5^d}$ be the Lagrange basis of $\calP_2(\sigma_\bfj)$ with the convention that $\rho_{\bfj,1}(\vertex(\bfj)) = 1$. By $\{\rho_{\bfj,k}^*\}_{k=1}^{5^d}$ denote its dual basis in the sense that
\begin{align}
  \label{eq:defdual}
  \bigskp{ \rho_{\bfj,k}}{\rho_{\bfj,m}^*}_{\sigma_\bfj} = \int_{\sigma_\bfj} \rho_{\bfj,k} \rho_{\bfj,m}^* \diff{x} = \delta_{k,m}.
\end{align}
Define $\Lambda_\bfj^* = \rho_{\bfj,1}^* \in L^\infty(\sigma_\bfj)$. We have thus constructed a family $\{\Lambda_\bfj^*\}_{\bfj \in \bfN^0}$. It turns out that this family is the dual basis for both $\{\Lambda_{h,\bfk}^\polQ\}_{\bfk \in \bfN^0}$ and $\{\Lambda_{h,\bfk}^\polP\}_{\bfk \in \bfN^0}$, as the following result shows.

\begin{lemma}[dual basis]
\label{lem:DualBasisProjection}
With the established notation we have that
\begin{align}
  \label{eq:dual-basis}
  \skp{\Lambda_{h,\bfm}^{\polQ}}{\Lambda_\bfj^*}_{\sigma_\bfj} =
  \skp{\Lambda_{h,\bfm}^{\polP}}{\Lambda_\bfj^*}_{\sigma_\bfj} = \delta_{\bfj,\bfm} \qquad \forall \bfj,\bfm \in \bfN^0.
\end{align}
where $\{\Lambda_{h,\bfk}^\polQ\}_{\bfk \in \bfN^0}$ and $\{\Lambda_{h,\bfk}^\polP\}_{\bfk \in \bfN^0}$ are, respectively, the Lagrange bases for $\Fespace^\polQ$ and $\Fespace^\polP$.
\end{lemma}
\begin{proof}
We begin with the case of cubes, that is $\Triang = \calQ_h$. If $\Lambda^{\polQ}_{h,\bfm|\sigma_\bfj} =0$, then both sides of \eqref{eq:dual-basis} are zero. Thus, we can assume in the following that $\Lambda^{\polQ}_{h,\bfm|\sigma_\bfj} \neq 0$. In this case, we can find coefficients $\alpha_{\bfm,k}$, $k = 1, \ldots, 5^d$ such that
\begin{align*}
  \Lambda^{\polQ}_{h,\bfm|\sigma_\bfj} = \sum_{k =1}^{5^d} \alpha_{\bfm,k} \rho_{\bfj,k}.
\end{align*}
Recall that we labeled the basis functions $\rho_{\bfj,k}$ such that $\rho_{\bfj,1}(\vertex(\bfj)) = 1$. Thus, if $\bfm=\bfj$, then $\alpha_{\bfm,1} =1$, and if $\bfm \neq \bfj$, then $\alpha_{\bfm,1} =0$. Thus, for $\bfm=\bfj$ we have
\begin{align*}
  \skp{\Lambda_{h,\bfm}^{\polQ}}{\Lambda_\bfj^*}_{\sigma_\bfj} = \sum_{k =1}^{5^d} \alpha_{\bfm,k} 
  \skp{\rho_{\bfj,k}}{\rho_{\bfj,1}^*}_{\sigma_\bfj} = \sum_{k =1}^{5^d} \alpha_{\bfm,k} \delta_{k,1}
  = \alpha_{\bfm,1} = \delta_{\bfj,\bfm}.
\end{align*}
This concludes the proof for $\mathcal{Q}_h$.

Notice now that the same proof also applies to the case of simplices, \ie for $\mathcal{T}_h$, which proves \eqref{eq:dual-basis}.
\end{proof}

With the notation of the previous proof we comment that, owing to the symmetries of $\sigma_\bfj$, it is possible to explicitly compute $\Lambda_\bfj^*$ in terms of $\{\rho_{\bfj,k}\}_{k=1}^{5^d}$, see Remark~\ref{rem:dual-basis1} below.

We define the operators $\PiQh\,:\, L^1(\Omega) \to\VQh$ and $\PiTh\,:\, L^1(\Omega) \to \VTh$ as
\begin{equation}
\label{eq:def-Pi}
  \PiQh w = \sum_{\bfj \in \bfN^0} \skp{w}{\Lambda_\bfj^*}_{\sigma_\bfj} \Lambda_{h,\bfj}^\polQ,
  \qquad 
  \PiTh w = \sum_{\bfj \in \bfN^0} \skp{w}{\Lambda_\bfj^*}_{\sigma_\bfj} \Lambda_{h,\bfj}^\polP.
\end{equation}
It follows from their definition and \eqref{eq:dual-basis} that these are linear projections.

Having defined our projections $\PiQh$ and $\PiTh$, we now study their stability. First, by symmetry of the mesh $\mathcal{T}_{h/2}$ we observe that the patches $\{\sigma_\bfj\}_{\bfj \in \bfN^0}$ and the corresponding dual basis functions $\{\Lambda_\bfj^*\}_{\bfj \in \bfN^0}$ are just translates of each other, respectively. Moreover, each patch, as a set and also as the union of simplices, is symmetric with respect to reflections along the coordinate directions with reflection point $\vertex(\bfj)$. Hence, the dual basis functions $\Lambda_\bfj^*$ are also symmetric with respect to reflection in the coordinate directions.

Using a scaling argument we see that
\begin{align*}
  \norm{\Lambda_\bfj^*}_{L^\infty(\sigma_\bfj)} &\lesssim \abs{\sigma_\bfj}^{-1} \lesssim h^{-d}.
\end{align*}
This and the definition of $\PiTh$ imply that for all $S \in \mathcal{T}_h$ we have
\begin{align}
  \label{eq:PiThlocalstab1}
  \sup_S \abs{\PiTh w} &\lesssim \dashint_{\calN_S} \abs{w} \diff{x}.
\end{align}
It follows from Jensen's inequality  that for every $p \in [1,\infty)$ and all $S \in \mathcal{T}_h$
\begin{align}
  \label{eq:PiThlocalstab2}
  \left( \dashint_S \abs{\PiTh w}^p \diff{x}  \right)^{\frac1p} &\lesssim \left( \dashint_{\calN_S} \abs{w}^p \diff{x}  \right)^{\frac1p}.
\end{align}
Note that \eqref{eq:PiThlocalstab1} and \eqref{eq:PiThlocalstab2} also hold with $\PiTh$ replaced by $\PiQh$ for all $Q \in \mathcal{Q}_h$.

Let us now turn to the proof of \eqref{eq:BNS}. In order to avoid different arguments close to the boundary $\partial \Omega$ we will first apply a simple reflection argument. Indeed, if $w \in W^{1,1}_0(\Omega)$ we can use repeated odd reflections to extend $w$ to $\bar{w} : (-1,2)^d \to \setR$. For example, if $x \in (0,1)^d$  and $\vare_1,\dots, \vare_d \in \set{-1,1}$ we set
\begin{align*}
  \bar{w}(\vare_1 x_1, \ldots, \vare_dx_d) = (-1)^{\vare_1+\cdots+\vare_d} w(x_1, \ldots, x_d).
\end{align*}
Let $\overline{\calT_h}$ denote an extension of $\mathcal{T}_h$ to $(-1,2)^d$. This allows us to define $\sigma_\bfj$ and $\Lambda_\bfj^*$ for boundary nodes $\bfj \in \bfN \setminus \bfN^0$. Since the dual basis functions  $\Lambda_\bfj^*$ are even with respect to the coordinate axes centered at $\vertex(\bfj)$ and $\bar{w}$ is extended by odd reflection, we have
\begin{align*}
  \skp{\bar{w}}{\Lambda_\bfj^*}_{\sigma_\bfj} = 0 \qquad \forall \bfj \in \bfN \setminus \bfN^0.
\end{align*}
In particular,
\begin{align*}
  \PiTh w = \sum_{\bfj \in \bfN} \skp{\bar{w}}{\Lambda_\bfj^*}_{\sigma_\bfj} \Lambda_{h,\bfj}^{\mathbb{P}}, \qquad \PiQh w = \sum_{\bfj \in \bfN} \skp{\bar{w}}{\Lambda_\bfj^*}_{\sigma_\bfj} \Lambda_{h,\bfj}^{\mathbb{Q}}.
\end{align*}
As a consequence, the orthotropic Sobolev stability near the boundary is just a special case of the interior stability for the extended objects. This certainly simplifies the calculations.

We consider first the case of simplices, \ie $\Triang = \calT_h$.

\begin{prop}[orthotropic stability of $\PiTh$]
\label{prop:OrthoStabPKuhnSimplices}
For every simplex $S \in \calT_h$ and all $k \in \{1, \ldots, d\}$, the projection $\PiTh$, defined in \eqref{eq:def-Pi}, satisfies
\begin{align}
\label{eq:local-stab1}
  \int_S \abs{\partial_k \PiTh w} \diff{x} &\lesssim \int_{\calN_S} \abs{\partial_k w}\diff{x}, \qquad \forall w \in W^{1,1}_0(\Omega).
\end{align}
\end{prop}
\begin{proof}
Recall that every $S \in \mathcal{T}_h$. is a Kuhn simplex with $d$ edges parallel to the coordinate axes. Let then $\gamma_k$ be the edge that is parallel to the $k$--th coordinate direction $\bfe_k$. Thus, we can find $\bfm \in \bfN^0$ with $\gamma_k = [\bfm,\bfm+ h \bfe_k]$. This shows that
\begin{align*}
  \partial_k \Lambda_{h,\bfj|_S}^\polP &=
    \begin{dcases}
      h^{-1} & \bfj = \bfm +\bfe_k, \\
      -h^{-1} & \bfj = \bfm, \\
      0 & \text{otherwise}.
    \end{dcases}
\end{align*}
Consequently,
\begin{align*}
  \partial_k (\PiTh w)_{|S}
  &= h^{-1}\skp{\bar{w}}{\Lambda_{\bfm+ \bfe_k}^*}_{\sigma_{\bfm+ \bfe_k}} - h^{-1}\skp{\bar{w}}{\Lambda_{\bfm}^*}_{\sigma_\bfm}
  \\
  &= h^{-1}\int_{\RRd} \bar{w} \left(\indicator_{\sigma_{\bfm+\bfe_k}} \Lambda_{\bfm+\bfe_k}^* - \indicator_{\sigma_{\bfm}} \Lambda_{\bfm}^* \right)\diff{x}.
\end{align*}
Recall now that $\Lambda_{\bfm+  \bfe_k}^*$ is a translate of $\Lambda_\bfm^*$ by $h\bfe_k$. Thus, every line integral in the direction of $x_k$ is zero, \ie
\begin{align*}
  \int_{-\infty}^\infty \left(\indicator_{\sigma_{\bfm+\bfe_k}} \Lambda_{\bfm+\bfe_k}^* - \indicator_{\sigma_{\bfm}} \Lambda_{\bfm}^* \right)\diff{x_k}= 0.
\end{align*}
This means that there is a function $g: \RRd \to \setR$ with compact support and weakly differentiable in the $x_k$ direction which satisfies
\begin{align*}
  \partial_k g = \indicator_{\sigma_{\bfm+\bfe_k}} \Lambda_{\bfm+\bfe_k}^* - \indicator_{\sigma_{\bfm}} \Lambda_{\bfm}^*.
\end{align*}
This proves that,
\begin{align*}
  \partial_k (\PiTh w)_{|S} &= h^{-1} \int_{\RRd} \bar{w} \partial_k g \diff{x} = - h^{-1}\int_{\RRd} \partial_k \bar{w} g \diff{x}.
\end{align*}
Since $\norm{\Lambda_{\bfm+\bfe_k}^*}_{L^\infty(\RR^d)} = \norm{\rho_m^*}_{L^\infty(\RR^d)} \lesssim h^{-d}$ we have $\norm{g}_{L^\infty(\RR^d)} \lesssim h^{1-d}$. Hence,
\begin{align*}
  \sup_S \abs{\partial_k (\PiTh w)} &\lesssim h^{-1} h^{1-d} \int_{\sigma_{\bfm+\bfe_k} \cup \sigma_\bfm} \abs{\partial_k \bar{w}} \diff{x}
    \lesssim \dashint_{\calN_S} \abs{\partial_k w} \diff{x},
\end{align*}
from which \eqref{eq:local-stab1} expediently follows.
\end{proof}

Let us now prove \eqref{eq:BNS} for $\PiQh$.

\begin{prop}[orthotropic stability of $\PiQh$]
For every cube $Q \in \calQ_h$ and all $k \in \{1, \ldots, d\}$, the projection $\PiQh$, defined in \eqref{eq:def-Pi}, satisfies
\begin{align*}
  \int_Q \abs{\partial_k \PiQh w} \diff{x} &\lesssim \int_{\calN_Q} \abs{\partial_k w}\diff{x}, \qquad \forall w \in W^{1,1}_0(\Omega).
\end{align*}
\end{prop}
\begin{proof}
Let $\bfj \in \bfN$ and $Q = Q_\bfj \in \calQ_h$. Given $k \in \{1, \ldots, d\}$ the vertices of $Q_\bfj$ partition into two sets: those with indices $\bfN_\bfj^-$, and those with indices $\bfN_\bfj^+ = \bfN_\bfj^- + \bfe_k$. After a short calculation we realize that
\begin{align*}
  (\partial_k \PiQh w)|_Q   &=
    \sum_{\bfm \in \bfN_\bfj^-}  \left( \skp{\bar{w}}{\sigma_{\bfj+\bfe_k}}_{\sigma_{\bfj+\bfe_k}} - \skp{\bar{w}}{\sigma_{\bfj}}_{\sigma_{\bfj}}\right).
\end{align*}
The remaining steps are analogously to the proof of Proposition~\ref{prop:OrthoStabPKuhnSimplices} and we obtain that, for all $Q \in \mathcal{Q}_h$,
\begin{align*}
  \sup_Q \abs{\partial_k (\PiQh w)} &\lesssim  \dashint_{\calN_Q} \abs{\partial_k w} \diff{x}. \qedhere
\end{align*}
\end{proof}

Having proved the stability of $\PiQh$ we provide an interesting observation.

\begin{remark}[relation between projections]
There is an interesting relation between $\PiQh$ and $\PiTh$. Let $\mathcal{P}^{\mathbb{Q}}_{\mathbb{P}}\,:\, \VQh \to \VTh$ be the linear bijection defined by
\begin{align*}
  (\mathcal{P}^{\mathbb{Q}}_{\mathbb{P}} w_h)(\vertex(\bfj)) = w_h(\vertex(\bfj)) \qquad \forall \bfj \in \bfN^0,
\end{align*}
\ie it preserves the values at the interior nodes. Then
\begin{align}
  \label{eq:commutates}
  \PiTh &= \mathcal{P}^{\mathbb{Q}}_{\mathbb{P}} \circ \PiQh, &
  \PiQh &= \left(\mathcal{P}^{\mathbb{Q}}_{\mathbb{P}}\right)^{-1} \circ \PiTh.
\end{align}
It turns out that the mapping $\calP_\polP^\polQ$ itself is locally orthotropically stable, \ie for all $Q \in \mathcal{Q}_h$ and all $k \in \set{1,\dots, d}$ we have for $p\in[1,\infty]$
\begin{align}
  \label{eq:transfer-orthotropic}
  \left(\dashint_Q \abs{\partial_k (\PiQh w)}^p \diff{x}\right)^{\frac 1p} \eqsim
  \left( \dashint_Q \abs{\partial_k (\PiTh w)}^p \diff{x} \right)^{\frac 1p}.
\end{align}
It suffices to consider the case $p=1$. The other choices of $p \in (1,\infty]$ follow by inverse estimates.  Since 
$\dim \polV_{h|Q}^\polQ = \dim \polV_{h|Q}^\polP = 2^d$ 
this can be seen by a local, finite dimensional, argument. We just need to show that if one of the two sides of~\eqref{eq:transfer-orthotropic} is zero, then so is the other one. Let us then fix $k \in \set{1,\dots, d}$ and $Q \in \mathcal{Q}_h$.

Suppose that $\partial_k \PiQh w = 0$ on $Q$. Let now $S \subset Q$ be a simplex with $S \in \calT_h$. We have to show that $\partial_k \PiTh w = 0$ on~$S$. The simplex $S$ must necessarily have an edge $\gamma_k = [\vertex, \vertex+h\bfe_k]$ that is parallel to the $x_k$--axis.  Since $\partial_k  \PiQh w = 0$, we have by~\eqref{eq:commutates}
\begin{align*}
  (\PiTh w)(\vertex) = (\PiQh w)( \vertex ) = 
  (\PiQh w)( \vertex + \bfe_k) =    (\PiTh w)(\vertex+\bfe_k).
\end{align*}
Since $\PiTh w $ is linear on $S$, this implies $\partial_k \PiTh w =0$ on $S$ as desired. 

Suppose now that $\partial_k \PiTh w = 0$  on $Q$. We have to show that $\partial_k \PiQh w=0$ on $Q$. For each of the $2^{d-1}$ edges $[\vertex, \vertex+h\bfe_k]$ of $Q$ we have, again by \eqref{eq:commutates},
\begin{align*}
  (\PiQh w)(\vertex) = (\PiTh w)( \vertex ) = 
  (\PiTh w)( \vertex + \bfe_k) =    (\PiQh w)(\vertex+\bfe_k).
\end{align*}
Since $\PiQh w \in \bbQ_1$, this implies that $\partial_k \PiQh w = 0$ on all of~$Q$. This concludes~\eqref{eq:transfer-orthotropic}.
% Indeed, let $w_h \in \VQh$ and $Q \in \calQ_h$, so that $w_{h|Q} \in \polQ_1$. Let now $S \subset Q$ be a simplex with $S \in \calT_h$. Fix now $i \in \{1,\ldots,d\}$. The simplex $S$ must necessarily have an edge $\gamma_i = [\vertex, \vertex+h\bfe_i]$ that is parallel to the $\bfe_i$ axis. We then see that, if $(\partial_i w_h)_{|Q} = 0$, then we must have
% \[
%   \calP_\polP^\polQ w_h( \vertex ) = w_h( \vertex ) = w_h( \vertex + h \bfe_i ) = \calP_\polP^\polQ w_h( \vertex + h \bfe_i )
% \]
% which, since $(\calP_\polP^\polQ w_h )_{|S} \in \polP_1$, it suffices to guarantee that $(\partial_i \calP_\polP^\polQ w_h )_{|S} = 0$ and thus the stability follows.
\end{remark}

% A = 1/180 * [ [12 , 3 , 2 , 3, 1 , 2 ]; [ 3 , 2 , 3 , 1 , 1 , 1 ]; [ 2 , 3 , 12 , 1 , 3 , 2 ]; [ 3 , 1 , 1 , 2 , 1 , 3 ]; [ 1 , 1 , 3 , 1 , 2 , 3 ]; [ 2 , 1 , 2 , 3 , 3 , 12]];

We conclude the discussion on the projection with some explicit computations.

\begin{remark}[dual basis functions]
\label{rem:dual-basis1}
It is possible to explicitly compute the dual basis functions $\Lambda_\bfj^*$ by inverting the mass matrix of $\mathcal{P}_2(\sigma_\bfj)$. Let us do this for $d=2$. By symmetry it suffices to find $\Lambda_\bfj^*$ in one of the eight (8) triangles that constitute $\sigma_\bfj$. We do this in local coordinates. Let $\rho_0$ denote the Lagrange basis function corresponding to vertex $\vertex(\bfj)$. Let $\rho_1$ and $\rho_2$ denote those of the other vertices. Finally, let $\rho_{\bfj,k}$ denote the ones corresponding to the midpoints, with the convention that $\rho_{\bfj,k}$ is associated to the midpoint opposite the vertex for $\rho_k$. Then, locally, we have
\[
  \Lambda_\bfj^* = \frac{1}{h^2} \left( 36 \rho_0 + 6 \rho_1 + 6 \rho_2 + 6 \rho_{\bfj,0} - \frac32 \rho_{\bfj,1} - \frac32 \rho_{\bfj,2} \right)
\]
\end{remark}

\begin{remark}[cubic meshes]
\label{rem:dual-basis2}
In the case of cubes it is not necessary to refine~$\mathcal{Q}_h$ to $\mathcal{T}_{h/2}$. Instead it is possible to define $\sigma_\bfj$ as the union of all the cubes that contain the node $\vertex(\bfj)$ and with a locally bilinear basis. In that case $\Lambda_\bfj^*$ is also locally bilinear. By symmetry, again, it suffices to calculate $\Lambda_\bfj^*$ in one of these cubes. Let us do so in two dimensions ($d=2$). We let $\rho_{00}$ denote the Lagrange basis function associated with vertex $\vertex(\bfj)$, $\rho_{01}$ and $\rho_{10}$ those corresponding to neighboring nodes, and $\rho_{11}$ the one for the node that is opposite along a diagonal. Then
\begin{align*}
  \Lambda_\bfj^* = \frac{16}{h^2} \left( 4 \rho_{00} - 2 \rho_{01} - 2 \rho_{10} + \rho_{11} \right).
\end{align*}
\end{remark}

\section{Numerical experiments}
\label{sec:Numerics}

Here we present a series of numerical examples aimed at illustrating our results and exploring the limits of our theory. The computations with a simplicial mesh have been carried out with the help of the \texttt{FreeFem++} package \cite{MR3043640}, whereas those with quad meshes use the \texttt{deal.II} library~\cite{dealII95}.

Let us now describe the employed solution procedure for the ensuing nonlinear systems of equations. We first realize that we may write $A_i(t) = B_i(t) t$. For instance, if $A_i(t) = |t|^{p_i-2}t$ then $B_i(t)=|t|^{p_i-2}$. With this splitting we implement the following preconditioned semi--implicit gradient flow. Starting from $u_h^0 \in \Fespace$, for $k \geq 0$, we let $u_h^{k+1} \in \Fespace$ be such that, for all $v_h \in \Fespace$,
\[
  \int_\Omega \GRAD\left( \frac{u_h^{k+1} - u_h^k}\dt \right) \cdot \GRAD v_h \diff{x} + \int_\Omega \sum_{i=1}^d B_i(\partial_i u_h^k) \partial_i u_h^{k+1} \partial_i v_h \diff{x} = \int_\Omega f v_h \diff{x}.
\]
The pseudo--time--step $\dt >0$ is chosen sufficiently small. We iterate until the \emph{energy}
\[
  \calJ( u_h^{k+1} - u_h^k ) = \sum_{i=1}^d \int_\Omega \varphi_i\left( \left|\partial_i( u_h^{k+1} - u_h^k ) \right| \right) \diff{x},
\]
is smaller than a prescribed tolerance.
More efficient methods, like relaxed Ka\v canov iterations \cite{DFTW20,BDS23}, are currently under investigation.

To be able to explore rates of convergence we need an explicit solution to a particular problem. We will consider, for $d =2$, the orthotropic $\bfp$--Laplacian of Example~\ref{ex:OrthoPLap}. We set $q_i = \tfrac{p_i}{p_i-1}$ to be their H\"older conjugate exponents. A direct computation shows that
\begin{equation}
\label{eq:ExplicitSolution}
  u(x_1,x_2) = \frac1{q_1}|x_1|^{q_1} - \frac1{q_2}|x_2|^{q_2}
\end{equation}
satisfies $-\Delta_\bfp u = 0$.

\subsection{Triangular meshes}
\label{sub:TriangleMeshes}

Here we consider the case of triangular meshes. We will consider several cases that illustrate our theory, as well as some that go beyond it.

\subsubsection{Right triangular mesh and smooth solution}
\label{subsub:StructuredSmooth}

\begin{table}
\begin{center}
  \begin{tabular}{|r||r|r||r|r|}
    \hline
    \hline
      $\dim \Fespace$ & $\| \| \nabla(u-u_h) \|_{\ell^p} \|_{L^p(\Omega)}$ & rate & $\| \bfV(\nabla u) - \bfV(\nabla u_h) \|_{\bfL^2(\Omega)}$ & rate \\
    \hline
    \hline
      121 & 2.3505E-01 & --- & 1.7311E-01 & ---\\ 
       \hline 
      441 & 1.1772E-01 & -0.53 & 8.6589E-02 & -0.54\\ 
       \hline 
      1681 & 5.8887E-02 & -0.52 & 4.3299E-02 & -0.52\\ 
       \hline 
      6561 & 2.9447E-02 & -0.51 & 2.1650E-02 & -0.51\\ 
       \hline 
      25921 & 1.4724E-02 & -0.50 & 1.0825E-02 & -0.50\\ 
       \hline 
      103041 & 7.3619E-03 & -0.50 & 5.4127E-03 & -0.50\\ 
       \hline 
      410881 & 3.6810E-03 & -0.50 & 2.7063E-03 & -0.50\\ 
%        \hline 
%       1640961 & 1.8405E-03 & -0.50 & 1.3532E-03 & -0.50\\ 
     \hline 
     \hline
  \end{tabular}
\end{center}
\caption{Rates of convergence for the experiment of \S\ref{subsub:StructuredSmooth}. The mesh consists of right triangular meshes $\boxslash$, and the solution is smooth: $p_1 = p_2 = \tfrac32$; see \eqref{eq:ExplicitSolution}.}
\label{tab:StructuredSmooth}
\end{table}

We set $\Omega = (0,1)^2$ and $p_1 = p_2 = p = \tfrac32$ and notice that, according to \eqref{eq:ExplicitSolution}, we have that $u \in W^{2,p}(\Omega)$. Our meshes consist of right triangles obtained by first subdividing in squares, which are then divided by their diagonals in the northeast direction. This results in the following configuration $\boxslash$. The results of the computation are presented in Table~\ref{tab:StructuredSmooth}. As expected, we have that the error both in norm and for $\bfV(\GRAD u)$ decay optimally, \ie as $\calO((\dim \Fespace)^{-\tfrac12})$.

\subsubsection{Right triangular mesh and non smooth solution}
\label{subsub:StructuredNonSmooth}

\begin{table}
\begin{center}
  \begin{tabular}{|r||r|r||r|r|}
    \hline
    \hline
      $\dim \Fespace$ & $\| \| \nabla(u-u_h) \|_{\ell^p} \|_{L^p(\Omega)}$ & rate & $\| \bfV(\nabla u) - \bfV(\nabla u_h) \|_{\bfL^2(\Omega)}$ & rate \\
    \hline
    \hline
      121 & 1.5081E-01 & --- & 1.6464E-01 & --- \\ 
      \hline
      441 & 8.6050E-02 & -0.43 & 8.3581E-02 & -0.52 \\ 
      \hline
      1681 & 4.8716E-02 & -0.43 & 4.2230E-02 & -0.51 \\ 
      \hline
      6561 & 2.7363E-02 & -0.42 & 2.1269E-02 & -0.50 \\ 
      \hline
      25921 & 1.5232E-02 & -0.43 & 1.0694E-02 & -0.50 \\ 
      \hline
      103041 & 8.6129E-03 & -0.41 & 5.3894E-03 & -0.50 \\ 
      \hline
      410881 & 5.6408E-03 & -0.31 & 2.7615E-03 & -0.48 \\ 
   \hline
   \hline
  \end{tabular}
\end{center}
\caption{Rates of convergence for the experiment of \S\ref{subsub:StructuredNonSmooth}. The mesh consists of right triangular meshes $\boxslash$, and the solution is non smooth: $p_1 = p_2 = 3$; see \eqref{eq:ExplicitSolution}.}
\label{tab:StructuredNonSmooth}
\end{table}

We set $\Omega = (0,1)^2$ and $p_1 = p_2 = p = 3$. Our meshes consist of right triangles as in Section~\ref{subsub:StructuredSmooth}. The results of the computation are presented in Table~\ref{tab:StructuredNonSmooth}, where we observe that the convergence rate for $\bfV(\GRAD u_h)$ is optimal, whereas it is not for the norm. This is due to the fact that, as we see from \eqref{eq:ExplicitSolution}, $u \notin W^{2,p}(\Omega)$. Notice, however, that
\[
  \partial_1^{1+s} u \eqsim x_1^{q-1-s},
\]
so that $u \in W^{1+s,p}(\Omega)$ for
\[
  p(q-1-s) > -1, \qquad \implies \qquad s< \frac56
\]
which coincides with the observed rate of convergence, \ie $\tfrac{5/6}2 \approx 0.42$. There is, however, a slight degradation in the rate of convergence at the finest mesh. This is likely due to quadrature error.

\subsubsection{Right triangular mesh and orthotropic solution}
\label{subsub:StructuredAniso}

\begin{table}
\begin{center}
  \begin{tabular}{|r||r|r||r|r||r|r|}
    \hline
    \hline
      $\dim \Fespace$ & $e_{p_1}$ & rate & $e_{p_2}$ & rate & $e_\bfV$ & rate \\
    \hline
    \hline
      121 & 1.2032E-01 & --- & 1.4809E-01 & --- & 1.6893E-01 & --- \\ 
     \hline 
      441 & 6.8595E-02 & -0.43 & 7.4169E-02 & -0.53 & 8.5105E-02 & -0.53 \\ 
     \hline 
      1681 & 3.8905E-02 & -0.42 & 3.7099E-02 & -0.52 & 4.2774E-02 & -0.51 \\ 
     \hline 
      6561 & 2.1990E-02 & -0.42 & 1.8551E-02 & -0.51 & 2.1465E-02 & -0.51 \\ 
     \hline 
      25921 & 1.2400E-02 & -0.42 & 9.2755E-03 & -0.50 & 1.0760E-02 & -0.50 \\ 
     \hline 
      103041 & 6.9782E-03 & -0.42 & 4.6377E-03 & -0.50 & 5.3894E-03 & -0.50 \\ 
     \hline 
      410881 & 3.9126E-03 & -0.42 & 2.3189E-03 & -0.50 & 2.6980E-03 & -0.50 \\ 
    \hline
    \hline
  \end{tabular}
\end{center}
\caption{Rates of convergence for the experiment of \S\ref{subsub:StructuredAniso}. The mesh consists of right triangular meshes $\boxslash$, and the solution is orthotropic: $p_1 =3$ and $ p_2 = \tfrac32$; see \eqref{eq:ExplicitSolution}.
\\
Legend: $e_{p_1} = \| \partial_1(u-u_h) \|_{L^{p_1}(\Omega)}$, $e_{p_2} = \| \partial_2(u-u_h) \|_{L^{p_2}(\Omega)}$, and $e_\bfV = \| \bfV(\nabla u) - \bfV(\nabla u_h) \|_{\bfL^2(\Omega)}$.}
\label{tab:StructuredAniso}
\end{table}

We now set $\Omega = (0,1)^2$, $p_1 = 3$, and $p_2 = \tfrac32$. Our meshes consist of right triangles as in Section~\ref{subsub:StructuredSmooth}. The results are presented in Table~\ref{tab:StructuredAniso}. Once again, the rate of convergence for $\bfV(\GRAD u_h)$ is optimal. Since the behavior, and smoothness, of the solution is now coordinate--dependent we also present the error in approximating each partial derivative in its corresponding norm. It is interesting to observe that the behavior of the error seems to correspond with the smoothness in each coordinate direction.

\subsubsection{Union Jack mesh}
\label{subsub:UnionJack}

\begin{table}
\begin{center}
  \begin{tabular}{|r||r|r||r|r||r|r|}
    \hline
    \hline
      $\dim \Fespace$ & $e_{p_1}$ & rate & $e_{p_2}$ & rate & $e_\bfV$ & rate \\
    \hline
    \hline
    121 & 1.2088E-01 & --- & 1.4766E-01 & --- & 1.6894E-01 & --- \\ 
     \hline 
    441 & 6.8459E-02 & -0.44 & 7.4240E-02 & -0.53 & 8.5104E-02 & -0.53 \\ 
     \hline 
    1681 & 3.8872E-02 & -0.42 & 3.7110E-02 & -0.52 & 4.2774E-02 & -0.51 \\ 
     \hline 
    6561 & 2.1982E-02 & -0.42 & 1.8553E-02 & -0.51 & 2.1465E-02 & -0.51 \\ 
     \hline 
    25921 & 1.2400E-02 & -0.42 & 9.2758E-03 & -0.50 & 1.0760E-02 & -0.50 \\ 
     \hline 
    103041 & 6.9829E-03 & -0.42 & 4.6378E-03 & -0.50 & 5.3895E-03 & -0.50 \\ 
     \hline 
    410881 & 3.9276E-03 & -0.42 & 2.3189E-03 & -0.50 & 2.6981E-03 & -0.50 \\ 
     \hline 
    \hline
  \end{tabular}
\end{center}
\caption{Rates of convergence for the experiment of \S\ref{subsub:UnionJack}. The mesh consists of a Union Jack pattern $\mathrlap{\boxplus}\boxtimes$, and the solution is orthotropic: $p_1 =3$ and $ p_2 = \tfrac32$; see \eqref{eq:ExplicitSolution}.
\\
Legend: $e_{p_1} = \| \partial_1(u-u_h) \|_{L^{p_1}(\Omega)}$, $e_{p_2} = \| \partial_2(u-u_h) \|_{L^{p_2}(\Omega)}$, and $e_\bfV = \| \bfV(\nabla u) - \bfV(\nabla u_h) \|_{\bfL^2(\Omega)}$.}
\label{tab:UnionJack}
\end{table}

In this experiment, we keep $\Omega = (0,1)^2$, $p_1 = 3$, and $p_2 = \tfrac32$. In this case, however, our meshes are obtained by first uniformly subdividing our domain into squares and then each square is subdivided into eight triangles by first cutting the square along both diagonals, and then by joining the midpoints of opposite sides. This results in every square having the following configuration: $\mathrlap{\boxplus}\boxtimes$. Since every triangle has two sides parallel to the coordinate axes, this setting does fit into the framework of Section~\ref{sec:interpolant}.

The results of our computations are presented in Table~\ref{tab:UnionJack}. Once again, $\bfV(\GRAD u_h)$ converges optimally.

\subsubsection{A structured mesh}
\label{subsub:StAndrewCross}

\begin{table}
\begin{center}
  \begin{tabular}{|r||r|r||r|r||r|r|}
    \hline
    \hline
      $\dim \Fespace$ & $e_{p_1}$ & rate & $e_{p_2}$ & rate & $e_\bfV$ & rate \\
    \hline
    \hline
    221 & 1.2109E-01 & --- & 1.4803E-01 & --- & 1.6895E-01 & --- \\ 
    \hline 
    841 & 6.8985E-02 & -0.42 & 7.4125E-02 & -0.52 & 8.5111E-02 & -0.51 \\ 
    \hline 
    3281 & 3.9112E-02 & -0.42 & 3.7074E-02 & -0.51 & 4.2777E-02 & -0.51 \\ 
    \hline 
    12961 & 2.2103E-02 & -0.42 & 1.8538E-02 & -0.50 & 2.1466E-02 & -0.50 \\ 
    \hline 
    51521 & 1.2463E-02 & -0.42 & 9.2690E-03 & -0.50 & 1.0760E-02 & -0.50 \\ 
    \hline 
    205441 & 7.0175E-03 & -0.42 & 4.6345E-03 & -0.50 & 5.3896E-03 & -0.50 \\ 
    \hline 
    820481 & 3.9466E-03 & -0.42 & 2.3172E-03 & -0.50 & 2.6982E-03 & -0.50 \\ 
    \hline 
    \hline
  \end{tabular}
\end{center}
\caption{Rates of convergence for the experiment of \S\ref{subsub:StAndrewCross}. The mesh is structured and with the following pattern $\boxtimes$, and the solution is orthotropic: $p_1 =3$ and $ p_2 = \tfrac32$; see \eqref{eq:ExplicitSolution}.
\\
Legend: $e_{p_1} = \| \partial_1(u-u_h) \|_{L^{p_1}(\Omega)}$, $e_{p_2} = \| \partial_2(u-u_h) \|_{L^{p_2}(\Omega)}$, and $e_\bfV = \| \bfV(\nabla u) - \bfV(\nabla u_h) \|_{\bfL^2(\Omega)}$.}
\label{tab:StAndrewCross}
\end{table}

The setting here is as before, but the mesh is now obtained by uniformly subdividing our domain into squares and then each square is subdivided into four triangles by cutting the square along both diagonals. This results in every square having the following configuration: $\boxtimes$. Notice that, in this scenario, no triangle has more than one side parallel to the coordinate axes. Due to this fact, this setting does not fit into the framework of Section~\ref{sec:interpolant}.

The results of our computations are presented in Table~\ref{tab:StAndrewCross}. Once again, $\bfV(\GRAD u_h)$ converges optimally. Interestingly, we observe that each coordinate derivative is approximated at the rate dictated by its regularity. The smoothness, or lack thereof, of the other coordinate derivative does not influence the rate of convergence.

\subsubsection{An unstructured mesh}
\label{subsub:Unstructured}

\begin{table}
\begin{center}
  \begin{tabular}{|r||r|r||r|r||r|r|}
    \hline
    \hline
      $\dim \Fespace$ & $e_{p_1}$ & rate & $e_{p_2}$ & rate & $e_\bfV$ & rate \\
    \hline
    \hline
    142 & 1.2748E-01 & --- & 1.4203E-01 & --- & 1.8365E-01 & --- \\ 
    \hline 
    515 & 7.2338E-02 & -0.44 & 7.1707E-02 & -0.53 & 9.2250E-02 & -0.53 \\ 
    \hline 
    1978 & 4.0030E-02 & -0.44 & 3.6072E-02 & -0.51 & 4.6482E-02 & -0.51 \\ 
    \hline 
    7695 & 2.2335E-02 & -0.43 & 1.7736E-02 & -0.52 & 2.3421E-02 & -0.50 \\ 
    \hline 
    30712 & 1.2326E-02 & -0.43 & 8.8891E-03 & -0.50 & 1.1665E-02 & -0.50 \\ 
    \hline 
    122163 & 6.8561E-03 & -0.42 & 4.4260E-03 & -0.51 & 5.8698E-03 & -0.50 \\ 
    \hline 
    490336 & 3.8249E-03 & -0.42 & 2.2233E-03 & -0.50 & 2.9463E-03 & -0.50 \\ 
    \hline 
    \hline
  \end{tabular}
\end{center}
\caption{Rates of convergence for the experiment of \S\ref{subsub:Unstructured}. The mesh is a Delaunay triangulation based on a uniform partition of each side of the square, and the solution is orthotropic: $p_1 =3$ and $ p_2 = \tfrac32$; see \eqref{eq:ExplicitSolution}.
\\
Legend: $e_{p_1} = \| \partial_1(u-u_h) \|_{L^{p_1}(\Omega)}$, $e_{p_2} = \| \partial_2(u-u_h) \|_{L^{p_2}(\Omega)}$, and $e_\bfV = \| \bfV(\nabla u) - \bfV(\nabla u_h) \|_{\bfL^2(\Omega)}$.}
\label{tab:Unstructured}
\end{table}

\begin{figure}
  \begin{center}
    \includegraphics[scale=0.12]{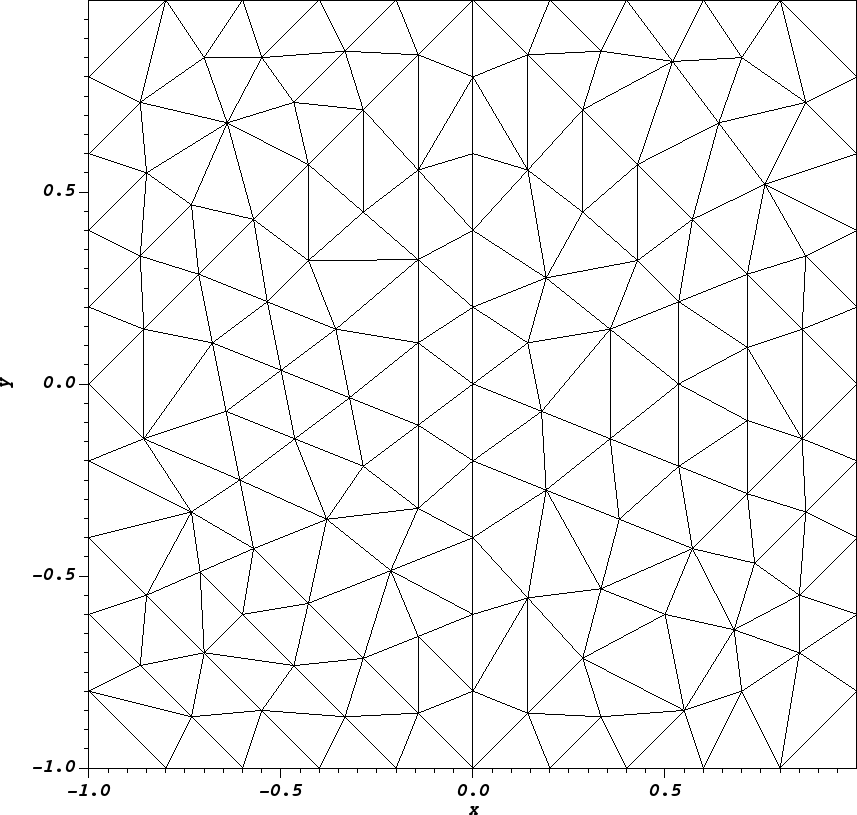}
    \includegraphics[scale=0.12]{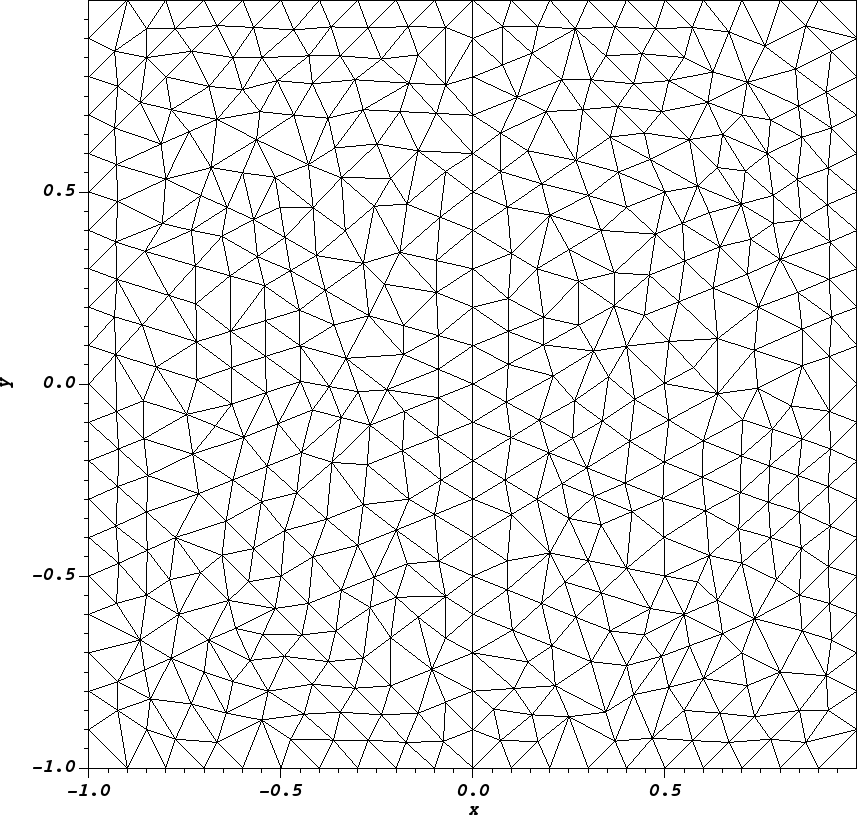}
    \includegraphics[scale=0.12]{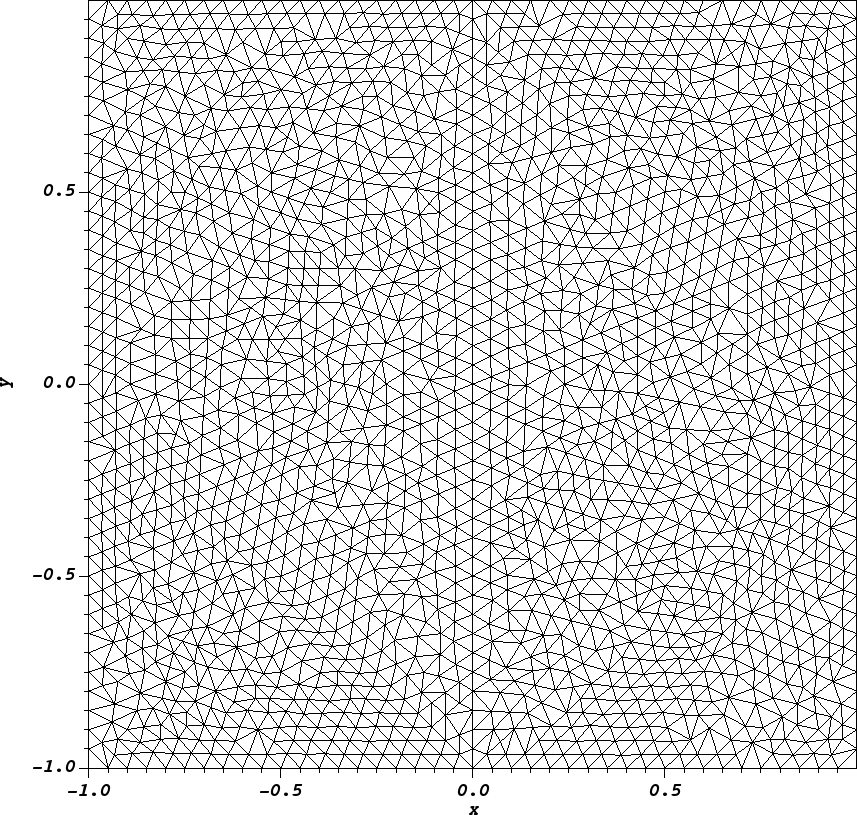}
  \end{center}
\caption{Some sample unstructured meshes for the experiments of Section~\ref{subsub:Unstructured}. A predetermined number of points is placed uniformly on each side of the square. On their basis, a Delaunay triangulation is constructed.}
\label{fig:UnstructuredMeshes}
\end{figure}

In this example we keep $\Omega = (0,1)^2$, $p_1 = 3$, and $p_2 = \tfrac32$. The mesh, in this case, is obtained by uniformly placing a number of points on each side of the square, and then constructing a Delaunay triangulation on the basis of these points. The mesh is, therefore, unstructured. Some sample meshes are presented in Figure~\ref{fig:UnstructuredMeshes}.

The results are presented in Table~\ref{tab:Unstructured}. As before, despite the fact that this mesh does not fit our theory, the rate of convergence for $\bfV(\GRAD u_h)$ is optimal. The same observations as before can be made regarding the approximation of coordinate derivatives.

\subsubsection{A circle}
\label{subsub:Circle}

\begin{table}
\begin{center}
  \begin{tabular}{|r||r|r||r|r||r|r|}
    \hline
    \hline
      $\dim \Fespace$ & $e_{p_1}$ & rate & $e_{p_2}$ & rate & $e_\bfV$ & rate \\
    \hline
    \hline
      15 & 1.1300E+00 & --- & 2.5032E+00 & --- & 2.1302E+00 & --- \\ 
       \hline 
      45 & 6.0409E-01 & -0.57 & 1.5565E+00 & -0.43 & 1.3393E+00 & -0.42 \\ 
       \hline 
      163 & 3.5001E-01 & -0.42 & 8.5932E-01 & -0.46 & 6.2778E-01 & -0.59 \\ 
       \hline 
      603 & 2.0171E-01 & -0.42 & 4.0154E-01 & -0.58 & 3.5075E-01 & -0.44 \\ 
       \hline 
      2348 & 1.0500E-01 & -0.48 & 2.0799E-01 & -0.48 & 1.6923E-01 & -0.54 \\ 
       \hline 
      8974 & 5.5161E-02 & -0.48 & 1.0502E-01 & -0.51 & 8.7200E-02 & -0.49 \\ 
       \hline 
      36011 & 3.0611E-02 & -0.42 & 5.1106E-02 & -0.52 & 4.1230E-02 & -0.54 \\ 
%        \hline 
%       142581 & 1.7410E-02 & -0.41 & 2.5325E-02 & -0.51 & 2.2629E-02 & -0.44 \\ 
    \hline 
    \hline
  \end{tabular}
\end{center}
\caption{Rates of convergence for the experiment of \S\ref{subsub:Circle}. The domain is the unit circle $\Omega = B_1$, and the solution is orthotropic: $p_1 =3$ and $ p_2 = \tfrac32$; see \eqref{eq:ExplicitSolution}.
\\
Legend: $e_{p_1} = \| \partial_1(u-u_h) \|_{L^{p_1}(\Omega)}$, $e_{p_2} = \| \partial_2(u-u_h) \|_{L^{p_2}(\Omega)}$, and $e_\bfV = \| \bfV(\nabla u) - \bfV(\nabla u_h) \|_{\bfL^2(\Omega)}$.}
\label{tab:Circle}
\end{table}

Let us set $\Omega = B_1$, \ie the unit circle, $p_1 = 3$, and $p_2 = \tfrac32$. The results are presented in Table~\ref{tab:Circle}. Despite the fact that this setting does not fit the assumptions of Section~\ref{sec:interpolant}, we observe that $\bfV(\GRAD u_h)$ does converge with an optimal rate. It is also interesting to observe that the coordinate derivatives still seem to converge at the independent rate dictated by their independent regularity.

\subsection{Quadrilateral meshes}
\label{sub:QuadMeshes}

We consider the exact solution that is given by \eqref{eq:ExplicitSolution} with $p_1 = 3$, and $p_2 = \tfrac32$.

\subsubsection{Structured mesh}
\label{subsub:StructuredQuads}

\begin{table}
\begin{center}
  \begin{tabular}{|r||r|r||r|r||r|r|}
    \hline
    \hline
      $\dim \Fespace$ & $e_{p_1}$ & rate & $e_{p_2}$ & rate & $e_\bfV$ & rate \\
    \hline
    \hline
    289 & 8.2750E-02 & --- & 9.2969E-02 & --- & 1.0616E-01 & --- \\ 
    \hline 
    1089 & 4.6972E-02 & -0.43 & 4.6515E-02 & -0.52 & 5.3392E-02 & -0.52 \\ 
    \hline 
    4225 & 2.6566E-02 & -0.42 & 2.3261E-02 & -0.51 & 2.6805E-02 & -0.51 \\ 
    \hline 
    16641 & 1.4989E-02 & -0.42 & 1.1631E-02 & -0.51 & 1.3440E-02 & -0.50 \\ 
    \hline 
    66049 & 8.4431E-03 & -0.42 & 5.8155E-03 & -0.50 & 6.7336E-03 & -0.50 \\ 
    \hline 
    263169 & 4.7505E-03 & -0.42 & 2.9078E-03 & -0.50 & 3.3715E-03 & -0.50 \\ 
    \hline 
    \hline
  \end{tabular}
\end{center}
\caption{Rates of convergence for the experiment of \S\ref{subsub:StructuredQuads}. The domain is  $\Omega = (0,1)^2$ and the mesh is uniform quadrilateral. The solution is orthotropic: $p_1 =3$ and $ p_2 = \tfrac32$; see \eqref{eq:ExplicitSolution}.
\\
Legend: $e_{p_1} = \| \partial_1(u-u_h) \|_{L^{p_1}(\Omega)}$, $e_{p_2} = \| \partial_2(u-u_h) \|_{L^{p_2}(\Omega)}$, and $e_\bfV = \| \bfV(\nabla u) - \bfV(\nabla u_h) \|_{\bfL^2(\Omega)}$.}
\label{tab:StructuredQuads}
\end{table}

Set $\Omega = (0,1)^2$ and we let the mesh be quadrilateral and uniform.
The results for a structured mesh are given in Table~\ref{tab:StructuredQuads}. As in the previous examples, $\bfV(\GRAD u_h)$ converges with an optimal rate. As before, each coordinate derivative converges at the rate dictated by its orthotropic regularity.

\subsubsection{A non simply connected domain}
\label{subsub:Donut}

\begin{figure}
  \begin{center}
    \includegraphics[scale=0.3]{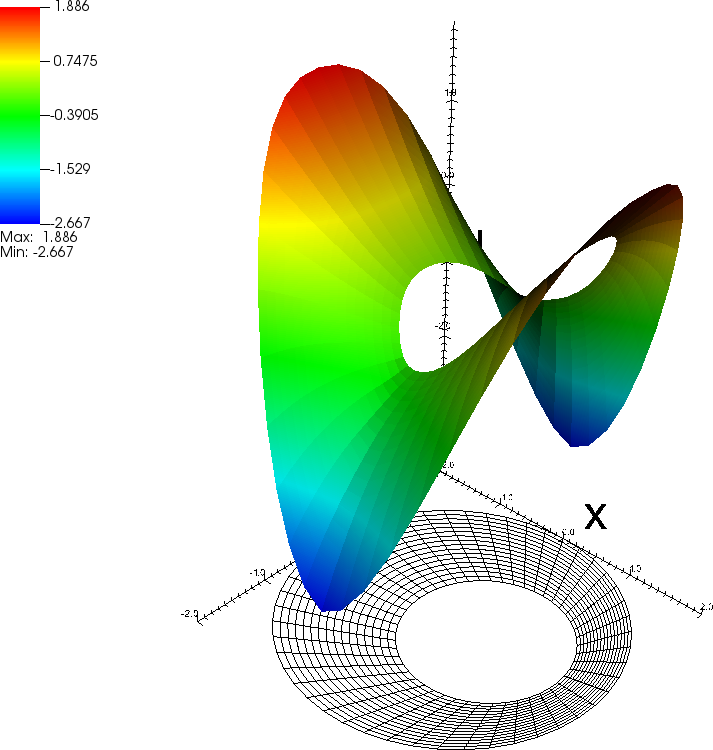}
  \end{center}
\caption{Sample mesh and computed solution for the experiment of \S\ref{subsub:Donut}.}
\label{fig:Donut}
\end{figure}

\begin{table}
\begin{center}
  \begin{tabular}{|r||r|r|r|r||r|r|}
    \hline
    \hline
      $\dim \Fespace$ & $e_{p_1}$ & rate & $e_{p_2}$ & rate & $e_\bfV$ & rate \\
    \hline
    \hline
    816 & 1.5662E-01 & --- & 2.0206E-01 & --- & 2.0456E-01 & --- \\ 
    \hline 
    3168 & 8.7980E-02 & -0.43 & 1.0141E-01 & -0.51 & 1.0247E-01 & -0.51 \\ 
    \hline 
    12480 & 4.9319E-02 & -0.42 & 5.0910E-02 & -0.50 & 5.1402E-02 & -0.50 \\ 
    49536 & 2.7667E-02 & -0.42 & 2.5547E-02 & -0.50 & 2.5791E-02 & -0.50 \\ 
    \hline 
    197376 & 1.5547E-02 & -0.42 & 1.2814E-02 & -0.50 & 1.2941E-02 & -0.50 \\ 
    \hline 
    787968 & 8.7554E-03 & -0.41 & 6.4248E-03 & -0.50 & 6.4924E-03 & -0.50 \\ 
    \hline 
    \hline
  \end{tabular}
\end{center}
\caption{Rates of convergence for the experiment of \S\ref{subsub:Donut}. The domain is  $\Omega = B_2(0)\setminus \overline{B_1(\tfrac12,0)}$ and the mesh is quadrilateral. The solution is orthotropic: $p_1 =3$ and $ p_2 = \tfrac32$; see \eqref{eq:ExplicitSolution}.
\\
Legend: $e_{p_1} = \| \partial_1(u-u_h) \|_{L^{p_1}(\Omega)}$, $e_{p_2} = \| \partial_2(u-u_h) \|_{L^{p_2}(\Omega)}$, and $e_\bfV = \| \bfV(\nabla u) - \bfV(\nabla u_h) \|_{\bfL^2(\Omega)}$.}
\label{tab:Donut}
\end{table}

As a final example we consider $\Omega = B_2(0)\setminus \overline{B_1(\tfrac12,0)}$, where $B_r(z)$ denotes the (open) ball of radius $r$ centered at $z \in \Real^d$. Notice that the domain is not simply connected. A plot of the mesh and the solution are presented in Figure~\ref{fig:Donut}. The results of the computation are summarized in Table~\ref{tab:Donut}.

%% New biber style  (Remove \i!!!)
\printbibliography 

@article {MR3378839,
    AUTHOR = {Bartels, S\"{o}ren and Nochetto, Ricardo H. and Salgado, Abner J.},
     TITLE = {A total variation diminishing interpolation operator and
              applications},
   JOURNAL = {Math. Comp.},
  FJOURNAL = {Mathematics of Computation},
    VOLUME = {84},
      YEAR = {2015},
    NUMBER = {296},
     PAGES = {2569--2587},
      ISSN = {0025-5718},
   MRCLASS = {65N30 (41A05 49J40 49M25 65D05 65N15)},
  MRNUMBER = {3378839},
MRREVIEWER = {Christian Wieners},
       DOI = {10.1090/mcom/2942},
       URL = {https://doi.org/10.1090/mcom/2942},
}

@article {MR2317830,
    AUTHOR = {Diening, L. and R{\r u}\v{z}i\v{c}ka, M.},
     TITLE = {Interpolation operators in {O}rlicz-{S}obolev spaces},
   JOURNAL = {Numer. Math.},
  FJOURNAL = {Numerische Mathematik},
    VOLUME = {107},
      YEAR = {2007},
    NUMBER = {1},
     PAGES = {107--129},
      ISSN = {0029-599X},
   MRCLASS = {65N30 (35J60 46E30 46E35 46N40)},
  MRNUMBER = {2317830},
       DOI = {10.1007/s00211-007-0079-9},
       URL = {https://doi.org/10.1007/s00211-007-0079-9},
}

@article {MR3043640,
    AUTHOR = {Hecht, F.},
     TITLE = {New development in freefem++},
   JOURNAL = {J. Numer. Math.},
  FJOURNAL = {Journal of Numerical Mathematics},
    VOLUME = {20},
      YEAR = {2012},
    NUMBER = {3-4},
     PAGES = {251--265},
      ISSN = {1570-2820},
   MRCLASS = {65Y15},
  MRNUMBER = {3043640},
       DOI = {10.1515/jnum-2012-0013},
       URL = {https://doi.org/10.1515/jnum-2012-0013},
}

@Article{dealII95,
  title   = {The \texttt{deal.II} Library, Version 9.5},
  author  = {Daniel Arndt and Wolfgang Bangerth and Maximilian Bergbauer and
             Marco Feder and Marc Fehling and Johannes Heinz and
             Timo Heister and Luca Heltai and Martin Kronbichler and
             Matthias Maier and Peter Munch and Jean-Paul Pelteret and
             Bruno Turcksin and David Wells and Stefano Zampini},
  journal = {Journal of Numerical Mathematics},
  year    = {2023},
  doi     = {10.1515/jnma-2023-0089},
  pages   = {231--246},
  volume  = {31},
  number  = {3}
}

@article {BDS23,
    AUTHOR = {Balci, Anna Kh. and Diening, Lars and Storn, Johannes},
     TITLE = {Relaxed {K}a\v{c}anov {S}cheme for the {$p$}-{L}aplacian with
              {L}arge {E}xponent},
   JOURNAL = {SIAM J. Numer. Anal.},
  FJOURNAL = {SIAM Journal on Numerical Analysis},
    VOLUME = {61},
      YEAR = {2023},
    NUMBER = {6},
     PAGES = {2775--2794},
      ISSN = {0036-1429,1095-7170},
   MRCLASS = {65N30 (35J70 35J92 65N22)},
  MRNUMBER = {4668389},
       DOI = {10.1137/22M1528550},
       URL = {https://doi.org/10.1137/22M1528550},
}

@article {DFTW20,
    AUTHOR = {Diening, L. and Fornasier, M. and Tomasi, R. and Wank, M.},
     TITLE = {A relaxed {K}a\v{c}anov iteration for the {$p$}-{P}oisson
              problem},
   JOURNAL = {Numer. Math.},
  FJOURNAL = {Numerische Mathematik},
    VOLUME = {145},
      YEAR = {2020},
    NUMBER = {1},
     PAGES = {1--34},
      ISSN = {0029-599X,0945-3245},
   MRCLASS = {65N30 (35J70 35J92 65N12 65N22)},
  MRNUMBER = {4091593},
MRREVIEWER = {Michael\ Neilan},
       DOI = {10.1007/s00211-020-01107-1},
       URL = {https://doi.org/10.1007/s00211-020-01107-1},
}

@article {MR2418205,
    AUTHOR = {Diening, Lars and Ettwein, Frank},
     TITLE = {Fractional estimates for non-differentiable elliptic systems
              with general growth},
   JOURNAL = {Forum Math.},
  FJOURNAL = {Forum Mathematicum},
    VOLUME = {20},
      YEAR = {2008},
    NUMBER = {3},
     PAGES = {523--556},
      ISSN = {0933-7741,1435-5337},
   MRCLASS = {35J55 (35D10 35J60)},
  MRNUMBER = {2418205},
MRREVIEWER = {Eugen\ Viszus},
       DOI = {10.1515/FORUM.2008.027},
       URL = {https://doi.org/10.1515/FORUM.2008.027},
}

@book {MR3024912,
    AUTHOR = {Pick, Lubo\v{s} and Kufner, Alois and John, Old\v{r}ich and
              Fu\v{c}\'{i}k, Svatopluk},
     TITLE = {Function spaces. {V}ol. 1},
    SERIES = {De Gruyter Series in Nonlinear Analysis and Applications},
    VOLUME = {14},
   EDITION = {extended},
 PUBLISHER = {Walter de Gruyter \& Co., Berlin},
      YEAR = {2013},
     PAGES = {xvi+479},
      ISBN = {978-3-11-025041-1},
   MRCLASS = {46Exx (46-02)},
  MRNUMBER = {3024912},
MRREVIEWER = {Dumitru\ Popa},
}

@book {MR0126722,
    AUTHOR = {Krasnosel'ski\u{i}, M. A. and Ruticki\u{i}, Ja. B.},
     TITLE = {Convex functions and {O}rlicz spaces.},
   EDITION = {Russian},
 PUBLISHER = {P. Noordhoff Ltd., Groningen, },
      YEAR = {1961},
     PAGES = {xi+249},
   MRCLASS = {46.35},
  MRNUMBER = {126722},
}

@book {MR2790542,
    AUTHOR = {Diening, Lars and Harjulehto, Petteri and H\"{a}st\"{o}, Peter
              and Ru\v{z}i\v{c}ka, Michael},
     TITLE = {Lebesgue and {S}obolev spaces with variable exponents},
    SERIES = {Lecture Notes in Mathematics},
    VOLUME = {2017},
 PUBLISHER = {Springer, Heidelberg},
      YEAR = {2011},
     PAGES = {x+509},
      ISBN = {978-3-642-18362-1},
   MRCLASS = {46-02 (26D10 31B15 35J60 35Q35 46E30 46E35 46N20)},
  MRNUMBER = {2790542},
MRREVIEWER = {Dorothee\ D.\ Haroske},
       DOI = {10.1007/978-3-642-18363-8},
       URL = {https://doi.org/10.1007/978-3-642-18363-8},
}

@article {MR0388811,
    AUTHOR = {Glowinski, R. and Marrocco, A.},
     TITLE = {Sur l'approximation, par \'{e}l\'{e}ments finis d'ordre un, et
              la r\'{e}solution, par p\'{e}nalisation-dualit\'{e}, d'une
              classe de probl\`emes de {D}irichlet non lin\'{e}aires},
   JOURNAL = {Rev. Fran\c{c}aise Automat. Informat. Recherche
              Op\'{e}rationnelle S\'{e}r. Rouge Anal. Num\'{e}r.},
  FJOURNAL = {Revue Fran\c{c}aise d'Automatique, Informatique et Recherche
              Op\'{e}rationnelle S\'{e}rie Rouge. Analyse Num\'{e}rique},
    VOLUME = {9},
      YEAR = {1975},
     PAGES = {41--76},
      ISSN = {0397-9342,2777-3493},
   MRCLASS = {65N35},
  MRNUMBER = {388811},
MRREVIEWER = {J.\ R.\ Cannon},
}

@article {MR1192966,
    AUTHOR = {Barrett, John W. and Liu, W. B.},
     TITLE = {Finite element approximation of the {$p$}-{L}aplacian},
   JOURNAL = {Math. Comp.},
  FJOURNAL = {Mathematics of Computation},
    VOLUME = {61},
      YEAR = {1993},
    NUMBER = {204},
     PAGES = {523--537},
      ISSN = {0025-5718,1088-6842},
   MRCLASS = {65N30},
  MRNUMBER = {1192966},
MRREVIEWER = {Qian\ Li},
       DOI = {10.2307/2153239},
       URL = {https://doi.org/10.2307/2153239},
}

@article {MR0672607,
    AUTHOR = {Tyukhtin, V. B.},
     TITLE = {The rate of convergence of approximation methods for solving
              one-sided variational problems. {I}},
   JOURNAL = {Teoret. Mat. Fiz.},
  FJOURNAL = {Akademiya Nauk SSSR. Teoreticheskaya i Matematicheskaya
              Fizika},
    VOLUME = {51},
      YEAR = {1982},
    NUMBER = {2},
     PAGES = {111--113, 121},
      ISSN = {0564-6162},
   MRCLASS = {49D15 (49A27)},
  MRNUMBER = {672607},
MRREVIEWER = {S.\ Raczy\'{n}ski},
}

@unpublished{antonini2023global,
      title={Global second-order estimates in anisotropic elliptic problems},
      author={Carlo Alberto Antonini and Andrea Cianchi and Giulio Ciraolo and Alberto Farina and Vladimir Maz'ya},
      year={2023},
      note={arXiv:2307.03052},
}

@article {MR3018035,
    AUTHOR = {Gao, Hongya and Liu, Chao and Tian, Hong},
     TITLE = {Remarks on a paper by {L}eonetti and {S}iepe},
   JOURNAL = {J. Math. Anal. Appl.},
  FJOURNAL = {Journal of Mathematical Analysis and Applications},
    VOLUME = {401},
      YEAR = {2013},
    NUMBER = {2},
     PAGES = {881--887},
      ISSN = {0022-247X,1096-0813},
   MRCLASS = {35J62 (35B65 35J25)},
  MRNUMBER = {3018035},
       DOI = {10.1016/j.jmaa.2012.12.037},
       URL = {https://doi.org/10.1016/j.jmaa.2012.12.037},
}

@article {MR2972430,
    AUTHOR = {Mercaldo, Anna and Rossi, Julio D. and Segura de Le\'{o}n,
              Sergio and Trombetti, Cristina},
     TITLE = {Behaviour of {$p$}-{L}aplacian problems with {N}eumann
              boundary conditions when {$p$} goes to 1},
   JOURNAL = {Commun. Pure Appl. Anal.},
  FJOURNAL = {Communications on Pure and Applied Analysis},
    VOLUME = {12},
      YEAR = {2013},
    NUMBER = {1},
     PAGES = {253--267},
      ISSN = {1534-0392,1553-5258},
   MRCLASS = {35J66 (35J75 35J92)},
  MRNUMBER = {2972430},
MRREVIEWER = {Ken\ Shirakawa},
       DOI = {10.3934/cpaa.2013.12.253},
       URL = {https://doi.org/10.3934/cpaa.2013.12.253},
}

@article {MR2878481,
    AUTHOR = {Leonetti, Francesco and Siepe, Francesco},
     TITLE = {Integrability for solutions to some anisotropic elliptic
              equations},
   JOURNAL = {Nonlinear Anal.},
  FJOURNAL = {Nonlinear Analysis. Theory, Methods \& Applications. An
              International Multidisciplinary Journal},
    VOLUME = {75},
      YEAR = {2012},
    NUMBER = {5},
     PAGES = {2867--2873},
      ISSN = {0362-546X,1873-5215},
   MRCLASS = {35J62 (35B65 35J25)},
  MRNUMBER = {2878481},
       DOI = {10.1016/j.na.2011.11.028},
       URL = {https://doi.org/10.1016/j.na.2011.11.028},
}

@article {MR3097236,
    AUTHOR = {Della Pietra, Francesco and Gavitone, Nunzia},
     TITLE = {Anisotropic elliptic equations with general growth in the
              gradient and {H}ardy-type potentials},
   JOURNAL = {J. Differential Equations},
  FJOURNAL = {Journal of Differential Equations},
    VOLUME = {255},
      YEAR = {2013},
    NUMBER = {11},
     PAGES = {3788--3810},
      ISSN = {0022-0396,1090-2732},
   MRCLASS = {35J60 (35A01 35A23 35B45 35B65)},
  MRNUMBER = {3097236},
       DOI = {10.1016/j.jde.2013.07.019},
       URL = {https://doi.org/10.1016/j.jde.2013.07.019},
}

@article {MR3119074,
    AUTHOR = {Boureanu, Maria-Magdalena and Udrea, Cristian and Udrea,
              Diana-Nicoleta},
     TITLE = {Anisotropic problems with variable exponents and constant
              {D}irichlet conditions},
   JOURNAL = {Electron. J. Differential Equations},
  FJOURNAL = {Electronic Journal of Differential Equations},
      YEAR = {2013},
     PAGES = {No. 220, 13},
      ISSN = {1072-6691},
   MRCLASS = {35J62 (35A01 35D30 35J20 35J25)},
  MRNUMBER = {3119074},
}

@incollection {MR2305342,
    AUTHOR = {Antontsev, S. and Shmarev, S.},
     TITLE = {Parabolic equations with anisotropic nonstandard growth
              conditions},
 BOOKTITLE = {Free boundary problems},
    SERIES = {Internat. Ser. Numer. Math.},
    VOLUME = {154},
     PAGES = {33--44},
 PUBLISHER = {Birkh\"{a}user, Basel},
      YEAR = {2007},
      ISBN = {978-3-7643-7718-2},
   MRCLASS = {35K55 (35D05 35K20)},
  MRNUMBER = {2305342},
       DOI = {10.1007/978-3-7643-7719-9\_4},
       URL = {https://doi.org/10.1007/978-3-7643-7719-9_4},
}

@article {MR3169023,
    AUTHOR = {Di Nardo, R. and Feo, F.},
     TITLE = {Existence and uniqueness for nonlinear anisotropic elliptic
              equations},
   JOURNAL = {Arch. Math. (Basel)},
  FJOURNAL = {Archiv der Mathematik},
    VOLUME = {102},
      YEAR = {2014},
    NUMBER = {2},
     PAGES = {141--153},
      ISSN = {0003-889X,1420-8938},
   MRCLASS = {35J62 (35A01 35A02 35D30)},
  MRNUMBER = {3169023},
       DOI = {10.1007/s00013-014-0611-y},
       URL = {https://doi.org/10.1007/s00013-014-0611-y},
}

@article {MR2968618,
    AUTHOR = {Di Castro, Agnese and P\'{e}rez-Llanos, Mayte and Urbano,
              Jos\'{e} Miguel},
     TITLE = {Limits of anisotropic and degenerate elliptic problems},
   JOURNAL = {Commun. Pure Appl. Anal.},
  FJOURNAL = {Communications on Pure and Applied Analysis},
    VOLUME = {11},
      YEAR = {2012},
    NUMBER = {3},
     PAGES = {1217--1229},
      ISSN = {1534-0392,1553-5258},
   MRCLASS = {35J62 (35J70)},
  MRNUMBER = {2968618},
       DOI = {10.3934/cpaa.2012.11.1217},
       URL = {https://doi.org/10.3934/cpaa.2012.11.1217},
}

@article {MR3115150,
    AUTHOR = {P\'{e}rez-Llanos, Mayte},
     TITLE = {Anisotropic variable exponent
              {$(p(\cdot),q(\cdot))$}-{L}aplacian with large exponents},
   JOURNAL = {Adv. Nonlinear Stud.},
  FJOURNAL = {Advanced Nonlinear Studies},
    VOLUME = {13},
      YEAR = {2013},
    NUMBER = {4},
     PAGES = {1003--1034},
      ISSN = {1536-1365,2169-0375},
   MRCLASS = {35J62 (35D30 35D40 35J20 35J70)},
  MRNUMBER = {3115150},
MRREVIEWER = {Florin\ Isaia},
       DOI = {10.1515/ans-2013-0414},
       URL = {https://doi.org/10.1515/ans-2013-0414},
}

@article {MR2503835,
    AUTHOR = {Di Castro, Agnese},
     TITLE = {Existence and regularity results for anisotropic elliptic
              problems},
   JOURNAL = {Adv. Nonlinear Stud.},
  FJOURNAL = {Advanced Nonlinear Studies},
    VOLUME = {9},
      YEAR = {2009},
    NUMBER = {2},
     PAGES = {367--393},
      ISSN = {1536-1365,2169-0375},
   MRCLASS = {35J60 (35B65 35J25)},
  MRNUMBER = {2503835},
       DOI = {10.1515/ans-2009-0207},
       URL = {https://doi.org/10.1515/ans-2009-0207},
}

@article {MR2813447,
    AUTHOR = {Di Castro, Agnese},
     TITLE = {Anisotropic elliptic problems with natural growth terms},
   JOURNAL = {Manuscripta Math.},
  FJOURNAL = {Manuscripta Mathematica},
    VOLUME = {135},
      YEAR = {2011},
    NUMBER = {3-4},
     PAGES = {521--543},
      ISSN = {0025-2611,1432-1785},
   MRCLASS = {35J62 (35A01 35A35 35B45 35B65 35D30 35J25)},
  MRNUMBER = {2813447},
MRREVIEWER = {Barbara\ Brandolini},
       DOI = {10.1007/s00229-011-0431-3},
       URL = {https://doi.org/10.1007/s00229-011-0431-3},
}

@article {MR3048250,
    AUTHOR = {Costea, Nicu\c{s}or and Moro\c{s}anu, Gheorghe},
     TITLE = {A multiplicity result for an elliptic anisotropic differential
              inclusion involving variable exponents},
   JOURNAL = {Set-Valued Var. Anal.},
  FJOURNAL = {Set-Valued and Variational Analysis. Theory and Applications},
    VOLUME = {21},
      YEAR = {2013},
    NUMBER = {2},
     PAGES = {311--332},
      ISSN = {1877-0533,1877-0541},
   MRCLASS = {35R70 (35A01 35D30 35J62)},
  MRNUMBER = {3048250},
MRREVIEWER = {Patrick\ Winkert},
       DOI = {10.1007/s11228-012-0224-1},
       URL = {https://doi.org/10.1007/s11228-012-0224-1},
}

@article {MR3168616,
    AUTHOR = {Farina, Alberto and Valdinoci, Enrico},
     TITLE = {Gradient bounds for anisotropic partial differential
              equations},
   JOURNAL = {Calc. Var. Partial Differential Equations},
  FJOURNAL = {Calculus of Variations and Partial Differential Equations},
    VOLUME = {49},
      YEAR = {2014},
    NUMBER = {3-4},
     PAGES = {923--936},
      ISSN = {0944-2669,1432-0835},
   MRCLASS = {49K10 (35J20)},
  MRNUMBER = {3168616},
MRREVIEWER = {Marie-Fran\c{c}oise\ Bidaut-V\'{e}ron},
       DOI = {10.1007/s00526-013-0605-9},
       URL = {https://doi.org/10.1007/s00526-013-0605-9},
}

@article {MR2861754,
    AUTHOR = {Mokhtari, Fares},
     TITLE = {Anisotropic parabolic problems with {O}rlicz data},
   JOURNAL = {Math. Methods Appl. Sci.},
  FJOURNAL = {Mathematical Methods in the Applied Sciences},
    VOLUME = {34},
      YEAR = {2011},
    NUMBER = {17},
     PAGES = {2095--2102},
      ISSN = {0170-4214,1099-1476},
   MRCLASS = {35K55 (35A01 35D30 35K20)},
  MRNUMBER = {2861754},
       DOI = {10.1002/mma.1508},
       URL = {https://doi.org/10.1002/mma.1508},
}

@article {MR4377996,
    AUTHOR = {B\"{o}gelein, Verena and Duzaar, Frank and Marcellini, Paolo
              and Scheven, Christoph},
     TITLE = {Boundary regularity for elliptic systems with {$p,q$}-growth},
   JOURNAL = {J. Math. Pures Appl. (9)},
  FJOURNAL = {Journal de Math\'{e}matiques Pures et Appliqu\'{e}es.
              Neuvi\`eme S\'{e}rie},
    VOLUME = {159},
      YEAR = {2022},
     PAGES = {250--293},
      ISSN = {0021-7824,1776-3371},
   MRCLASS = {35B65 (35B45 35J57 35J60 49J10 49N60)},
  MRNUMBER = {4377996},
       DOI = {10.1016/j.matpur.2021.12.004},
       URL = {https://doi.org/10.1016/j.matpur.2021.12.004},
}

@article {MR3073153,
    AUTHOR = {B\"{o}gelein, Verena and Duzaar, Frank and Marcellini, Paolo},
     TITLE = {Parabolic systems with {$p,q$}-growth: a variational approach},
   JOURNAL = {Arch. Ration. Mech. Anal.},
  FJOURNAL = {Archive for Rational Mechanics and Analysis},
    VOLUME = {210},
      YEAR = {2013},
    NUMBER = {1},
     PAGES = {219--267},
      ISSN = {0003-9527,1432-0673},
   MRCLASS = {35K51 (35A01 35D30 35K59)},
  MRNUMBER = {3073153},
MRREVIEWER = {Horst\ Heck},
       DOI = {10.1007/s00205-013-0646-4},
       URL = {https://doi.org/10.1007/s00205-013-0646-4},
}

@article {MR2978583,
    AUTHOR = {Tersenov, Ar. S.},
     TITLE = {New a priori estimates for solutions of anisotropic elliptic
              equations},
   JOURNAL = {Sibirsk. Mat. Zh.},
  FJOURNAL = {Rossi\u{i}skaya Akademiya Nauk. Sibirskoe Otdelenie. Institut
              Matematiki im. S. L. Soboleva. Sibirski\u{i}
              Matematicheski\u{i} Zhurnal},
    VOLUME = {53},
      YEAR = {2012},
    NUMBER = {3},
     PAGES = {672--686},
      ISSN = {0037-4474},
   MRCLASS = {35J62 (35B45 35J25 35J75)},
  MRNUMBER = {2978583},
MRREVIEWER = {Lubomira\ G.\ Softova},
       DOI = {10.1134/S0037446612020346},
       URL = {https://doi.org/10.1134/S0037446612020346},
}

@article {MR2483260,
    AUTHOR = {Antontsev, Stanislav and Chipot, Michel},
     TITLE = {Anisotropic equations: uniqueness and existence results},
   JOURNAL = {Differential Integral Equations},
  FJOURNAL = {Differential and Integral Equations. An International Journal
              for Theory \& Applications},
    VOLUME = {21},
      YEAR = {2008},
    NUMBER = {5-6},
     PAGES = {401--419},
      ISSN = {0893-4983},
   MRCLASS = {35J60 (35J25 46E35)},
  MRNUMBER = {2483260},
MRREVIEWER = {Niko\ M.\ Marola},
}

@phdthesis{XiaThesis,
  author  = "Xia Chao",
  title   = "On a class of anisotropic problems",
  school  = "Albert-Ludwigs-Universit{\"a}t Freiburg im Breisgau",
  year    = "2012",
  type    = "PhD Thesis",
  month   = {4}
}

@article {MR3438591,
    AUTHOR = {Kozhevnikova, L. M. and Khadzhi, A. A.},
     TITLE = {Existence of solutions of anisotropic elliptic equations with
              nonpower nonlinearities in unbounded domains},
   JOURNAL = {Mat. Sb.},
  FJOURNAL = {Matematicheski\u{i} Sbornik},
    VOLUME = {206},
      YEAR = {2015},
    NUMBER = {8},
     PAGES = {99--126},
      ISSN = {0368-8666,2305-2783},
   MRCLASS = {35J62 (35J25 46E30)},
  MRNUMBER = {3438591},
MRREVIEWER = {Ji\v{r}\'{i}\ R\'{a}kosn\'{i}k},
       DOI = {10.4213/sm8482},
       URL = {https://doi.org/10.4213/sm8482},
}

@article {MR4502898,
    AUTHOR = {Bildhauer, Michael and Fuchs, Martin},
     TITLE = {On the global regularity for minimizers of variational
              integrals: splitting-type problems in 2{D} and extensions to
              the general anisotropic setting},
   JOURNAL = {J. Elliptic Parabol. Equ.},
  FJOURNAL = {Journal of Elliptic and Parabolic Equations},
    VOLUME = {8},
      YEAR = {2022},
    NUMBER = {2},
     PAGES = {853--884},
      ISSN = {2296-9020,2296-9039},
   MRCLASS = {49N60 (49J45 49Q20)},
  MRNUMBER = {4502898},
MRREVIEWER = {Luca\ Lussardi},
       DOI = {10.1007/s41808-022-00179-4},
       URL = {https://doi.org/10.1007/s41808-022-00179-4},
}

@article {MR3798025,
    AUTHOR = {Bousquet, P. and Brasco, L. and Leone, C. and Verde, A.},
     TITLE = {On the {L}ipschitz character of orthotropic {$p$}-harmonic
              functions},
   JOURNAL = {Calc. Var. Partial Differential Equations},
  FJOURNAL = {Calculus of Variations and Partial Differential Equations},
    VOLUME = {57},
      YEAR = {2018},
    NUMBER = {3},
     PAGES = {Paper No. 88, 33},
      ISSN = {0944-2669,1432-0835},
   MRCLASS = {35J62 (35B65 35J70 49K20)},
  MRNUMBER = {3798025},
MRREVIEWER = {Teemu\ Lukkari},
       DOI = {10.1007/s00526-018-1349-3},
       URL = {https://doi.org/10.1007/s00526-018-1349-3},
}

@article {MR3749368,
    AUTHOR = {Bousquet, Pierre and Brasco, Lorenzo},
     TITLE = {{$C^1$} regularity of orthotropic {$p$}-harmonic functions in
              the plane},
   JOURNAL = {Anal. PDE},
  FJOURNAL = {Analysis \& PDE},
    VOLUME = {11},
      YEAR = {2018},
    NUMBER = {4},
     PAGES = {813--854},
      ISSN = {2157-5045,1948-206X},
   MRCLASS = {35J92 (35B65 49K20 49N60)},
  MRNUMBER = {3749368},
MRREVIEWER = {Shulin\ Zhou},
       DOI = {10.2140/apde.2018.11.813},
       URL = {https://doi.org/10.2140/apde.2018.11.813},
}

% \bibliographystyle{plain}
% \bibliography{biblio}

\end{document}